\documentclass[12pt]{article}

\usepackage{glossaries}

\usepackage[utf8]{inputenc}
\usepackage[T1]{fontenc}
\usepackage{lmodern}
\usepackage{xspace} 
\usepackage{bbm} 
\usepackage{soul} 

\usepackage{standalone}
\usepackage{tikz}
\usepackage{tkz-graph}
\usetikzlibrary{arrows,calc,fit,positioning,decorations.pathmorphing,decorations.pathreplacing,decorations.markings}
\usepackage{figstyle}

\newcommand{\parte}{\textup{part}_t}
\newcommand{\parteu}{\textup{part}}
\newcommand{\vol}{\textup{vol}_t}
\newcommand{\volu}{\textup{vol}}

\usepackage[english]{babel}
\usepackage{verbatim}

\usepackage{amsmath}
\allowdisplaybreaks
\usepackage{amssymb}
\usepackage{amscd}
\usepackage{amsfonts} 
\usepackage[normalem]{ulem} 
\usepackage{amsthm}
\usepackage{cases} 
\usepackage{mathtools}

\usepackage{graphicx} 
\usepackage{tabu}
\usepackage{epic,eepic}
\usepackage{color}   
\usepackage{pst-plot} 
\usepackage[space]{grffile} 
\usepackage{etoolbox} 

\usepackage{hyperref}   
\usepackage{url}

\usepackage{enumerate}

\newtheorem{thm}{Theorem}
\newtheorem{lem}[thm]{Lemma}
\newtheorem{cor}[thm]{Corollary}
\newtheorem{question}[thm]{Question}

\newtheorem{remark}[thm]{Remark}

\newtheorem{example}[thm]{Example}

\newtheorem{identity}[thm]{Identity}
\newtheorem*{claim-non}{Claim}

\newcommand{\x}{\,\,}

\newcommand{\ignore}[1]{{}}

\newcommand{\defi}[1]{\emph{#1}}

\newcommand{\calP}{\mathcal{P}}
\newcommand{\calR}{\mathcal{R}}
\newcommand{\calL}{\mathcal{L}}
\newcommand{\calQ}{\mathcal{Q}}
\newcommand{\calT}{\mathcal{T}}

\newcommand{\calX}{\mathcal{X}}

\newcommand{\volev}[1]{\textup{vol}_{{#1} \mathrm{\,even}}}
\newcommand{\volod}[1]{\textup{vol}_{{#1} \mathrm{\,odd}}}

\newcommand{\ZZeven}[2]{Z^{{#1} \mathrm{\,even}}_{{#2}}} 
\newcommand{\ZZodd}[2]{Z^{{#1} \mathrm{\,odd}}_{{#2}}} 

\newcommand{\Out}{\mathrm{Out}}
\newcommand{\In}{\mathrm{In}}

\newcommand{\rarr}{\rightarrow}

\newcommand{\floor}[1]{\lfloor{#1}\rfloor}
\newcommand{\Hzz}{\emptyset}
\newcommand{\um}{\mathbbm{1}}

\newcommand{\IR}{\mathbb{R}}
\newcommand{\ZZ}{\mathbb{Z}}

\newcommand{\Mod}[1]{\ (\mathrm{mod}\ #1)}

\title{On the period collapse of a family of \\ Ehrhart quasi-polynomials
\thanks{
   This work was partially supported by
   Conselho Nacional de Desenvolvimento Científico e Tecnológico – CNPq
   (Proc. 423833/2018-9, 308116/2016-0, and~456792/2014-7),
   grants \#2012/24597-3 and \#2013/03447-6,
   São Paulo Research Foundation (FAPESP), and by the grant INSMI-CNRS.}
}

\author{\hspace{1cm} Cristina G.~Fernandes
  \thanks{Instituto de Matem\'atica e Estat\'istica, Universidade de S\~ao Paulo, 05508-090 S\~ao Paulo, Brazil
    (\texttt{cris@ime.usp.br}, \texttt{coelho@ime.usp.br}, \texttt{srobins@ime.usp.br}).}
\and Jos\'e C.~de~Pina \footnotemark[2] \hspace{1cm}
\and Jorge~Luis Ram\'irez~Alfons\'in 
	\thanks{IMAG, Univ.\ Montpellier, CNRS, Montpellier, France and UMI2924 - Jean-Christophe Yoccoz, CNRS-IMPA (\texttt{jorge.ramirez-alfonsin@umontpellier.fr}).} 
\and Sinai Robins \footnotemark[2]
}

\makeglossaries

\begin{document}
\maketitle

\begin{abstract}
 A graph whose nodes have degree~$1$ or~$3$ is called a $\{1,3\}$-graph. 
  Liu and Osserman associated a polytope  to each $\{1,3\}$-graph and 
  studied the Ehrhart quasi-polynomials of these polytopes.  
  They showed that the vertices of these polytopes have coordinates 
  in the set $\{0,\frac14,\frac12,1\}$, which implies that the period 
  of their Ehrhart quasi-polynomials is either $1, 2$, or $4$.  
  We show that the period of the Ehrhart quasi-polynomial of these polytopes is at most~2 
  if the graph is a tree or a cubic graph, and it is equal to~$4$ otherwise.
  
  In the process of proving this theorem, several interesting combinatorial 
  and geometric properties of these polytopes were uncovered, arising from the structure of their associated graphs.
The tools developed here may find other applications in the study of Ehrhart quasi-polynomials and enumeration problems
for other polytopes that arise from graphs.
  Additionally, we have identified some interesting connections with triangulations of 3-manifolds. 
\end{abstract}

\newpage
\tableofcontents

\section{Introduction}

A \emph{$\{1,3\}$-graph} is a graph whose nodes have degree $1$ or~$3$.
Liu and Osserman~\cite{LiuO2006} associated a polytope $\calP_G$ to each 
$\{1,3\}$-graph~$G$ and studied the Ehrhart quasi-polynomial
arising from $\calP_G$.  
They were mainly motivated by the relation of these quasi-polynomials to 
the study of \emph{dormant torally indigenous bundles on a general curve}, 
objects arising in algebraic geometry~\cite{Mochizuki1996}.  
This connection was further investigated in~\cite{Wakabayashi2013}, and more 
properties of the polytope $\calP_G$ were presented in~\cite{FernandesPRAR2020}.

Specifically, Liu and Osserman~\cite[Theorem~3.9]{LiuO2006} observed that
Mochizuki~\cite{Mochizuki1996} implicitly proved that the value of the 
Ehrhart quasi-polynomial of $\calP_G$ on odd primes is the number of 
dormant torally indigenous bundles in a certain class of curves
parametrized 
by the number of nodes and edges of~$G$.  
Also, they had proved that the coordinates of 
all vertices of~$\calP_G$ are in~$\{0,\frac14,\frac12,1\}$.  
This implies that the period of the Ehrhart quasi-polynomial 
of~$\calP_G$ is either~1 or~2 or~4~\cite{BeckS2009,Ehrhart1977}.
Using a result of Mochizuki~\cite{Mochizuki1996}, they concluded that the
odd constituents of this Ehrhart quasi-polynomial are the same polynomial.  
Liu and Osserman~\cite[Question~4.3]{LiuO2006} then raised 
questions about the period of the Ehrhart quasi-polynomial of~$\calP_G$.  
In this paper, we  answer some of these questions. 


The polytope $\calP_G$ has nice geometric and combinatorial
properties. For instance, Wakabayashi has proved~\cite[Proposition~5.3
  and Corollary~5.4]{Wakabayashi2013} that, for cubic graphs $G_1$ and~$G_2$, 
the polytopes $\calP_{G_1}$ and $\calP_{G_2}$ are isomorphic
(that is, there is an $\IR$-linear bijection $f:\IR^d\rightarrow
\IR^d$ such that $\calP_{G_2} = f(\calP_{G_1})$) if and only if the
graphs $G_1$ and $G_2$ are isomorphic.  Let $T$ be a
$\{1,3\}$-tree. It turns out that $\calP_T$ enjoys even more plenty of
interesting geometric properties that can be combinatorially
described. The latter might provide useful and attractive insights for
appealing questions on {\em 0/1 polytopes} (that is, the convex hull
of subsets of~$\{0,1\}^d$, the vertices of the $d$-cube). All these
will be discussed throughout the paper.
 
Our work is also closely connected with some invariants of $3$-manifolds, 
investigated by Maria and Spreer~\cite{MariaS2016}.   
They associated a linear system of inequalities to a
triangulation~$T$ of a 3-manifold and studied \emph{admissible
  colourings} of the edges of $T$ with the aim to understand better
Turaev-Viro type invariants.  It turns out that admissible
colourings correspond to integer points belonging to a polytope related 
to one of the two central polytopes investigated in this paper.
This relationship as well as an application, with the same topological flavor as in~\cite{MariaS2016},
on a problem concerning non-intersecting closed curves in the plane
will be explained towards the end of the paper.

\subsection{Ehrhart quasi-polynomials and period collapse}

Let $\calL$ be a sublattice of $\ZZ^d$ and fix $u \in \ZZ^d$.
We define the \emph{discrete volume} of a polytope~$\calP$ in~$\IR^d$   
with respect to the coset $\calL+u \subseteq \ZZ^d$ by
\[
  \volu_{\calL+u}(\calP) := |\calP \cap (\calL + u)|.
\]
Here we use the standard notation for 
the translation of any set $Y \subseteq \IR^d$ by a fixed vector $u \in \IR^d$, namely
$Y + u := \{y + u : y \in Y\} \subseteq \IR^d$.
Ehrhart~\cite{Ehrhart1977} introduced the function 
\[
L_{\calP}(t) := \volu_{\ZZ^d}(t\calP), 
\]
which is the number of lattice points in the \emph{dilated polytope} $t\calP := \{ tx: x \in \calP\}$, 
for a nonnegative integer dilation $t$.
Ehrhart showed that if~$\calP$ is an integral polytope, then this function is a polynomial in the integer 
parameter~$t$. 
A \emph{quasi-polynomial} $f(t)$ is a function defined by
a list $p_0,p_1,\ldots,p_{\alpha-1}$ of polynomials
such that
\begin{numcases}{f(t)=}
  p_0(t) & \mbox{if $t \equiv 0 \Mod{\alpha}$,} \nonumber \\
  p_1(t) & \mbox{if $t \equiv 1 \Mod{\alpha}$,} \nonumber \\
  \quad \vdots & \quad \vdots  \nonumber \\
  p_{\alpha-1}(t) & \mbox{if $t \equiv \alpha-1 \Mod{\alpha}$.} \nonumber 
\end{numcases}
The minimal such $\alpha$ is the {\em period} of $f$ and
$p_0,p_1,\ldots,p_{\alpha-1}$ are the \emph{constituents} of $f$.
More generally, Ehrahrt also showed that if~$\calP$ is a rational polytope, 
then the function $L_{\calP}(t)$ is a quasi-polynomial, for integer values of $t$,
whose period divides 
the least common multiple of the denominators in the coordinates 
of the vertices of $\calP$~\cite{BeckS2009,Ehrhart1977}.

If a rational polytope $\calP$ has a quasi-polynomial $L_{\calP}(t)$
whose period is $p$, we also call $p$ the period of $\calP$.  The
denominator of $\calP$ is the minimal integer $\alpha$ such that the
vertices of the dilated polytope $\alpha \calP$ are integral.  For a
`generic' rational polytope, one expects its period to be equal to its
denominator.  Regarding the complexity of computing the periods,
Woods~\cite{Woods2005} has shown that, for fixed dimension $d$, there
is a polynomial-time algorithm which, given a rational polytope $\calP
\subset \mathbb R^d$ and an integer $n > 0$, decides whether~$n$ is a
multiple of the period of the quasi-polynomial $L_{\calP}(t)$.

When we have a rational polytope whose period is smaller than its
denominator, we refer to this situation as \emph{period collapse}.
Usually,
researchers~\cite{BeckSamWoods2008,McAllisterRochais2018,McAllisterW2017,Woods2005,Woods2015}
prove that the phenomenon of period collapse exists for some particular
rational polytope $\calP$ by exhibiting a decomposition of $\calP$
into rational simplices, and then applying a (different) unimodular
affine linear transformation to each simplex.
If, somewhat magically, the reassembling of all of the images of
these simplices form an integral polytope, then the Ehrhart quasi-polynomial of $\calP$ is in
fact a polynomial.  It is a well-known open problem whether it is
always possible to carry out this process, for any rational polytope
that possesses period collapse~\cite{Haase2008}.
 
Here we develop a different technique for proving period collapse,
building on an idea of Liu and Osserman.  Namely, rather than
decomposing the object into smaller polytopes, and using special types
of unimodular transformations for each of them, we instead decompose
the integer lattice into a certain sublattice $\cal L$, together with
all of its cosets in the integer lattice.  Then, for each fixed coset
of $\cal L$, we count the number of points of $\cal L$ that belong to
$\calP$, 
and show that this number is in bijection with the number of integer
points that belong to another, naturally-defined, integral polytope.
This novel technique seems interesting in itself and might be useful in a wider context.

In this research, we have made extensive use of both of the following software packages: LattE~\cite{latte}, and \texttt{polymake}~\cite{polymake_XML:ICMS_2016, polymake:ICMS_2006}.

\subsection{Liu and Osserman's polytopes}

Let $G$ be a $\{1,3\}$-graph.  We say that a node of degree~1 is a
\emph{leaf node} and a node of degree~3 is an \emph{internal node}.  We denote by
$V(G)$, $E(G)$, and~$I(G)$ the set of nodes, the set of edges, and
the set of internal nodes of $G$, respectively. If the graph is
clear from the context we write simply $V$, $E$, and $I$.
A subgraph $H$ of $G$ is \emph{internally Eulerian} 
if the degree in $H$ of every internal node
of~$G$ is equal to zero or two.
In particular, the empty subgraph is internally Eulerian. 

If $X$ is a finite set, then we denote by  
$\um_S : X \rarr \{0,1\}$ the characteristic vector of a set~$S \subseteq X$.
If $S = \{e\}$ for some $e$ in $X$, we write $\um_e$.
If $S = E(H)$ for some graph $H$, we write $\um_H$.
We allow all of our graphs to have loops and parallel edges.

Liu and Osserman~\cite{LiuO2006} associated to each $\{1,3\}$-graph $G$ a
polytope $\calP_G$ in $\IR^E$  as follows.
For each internal node $v$ of $G$, let $a$,~$b$, and $c$ be the three edges incident to $v$.  
Denote by~$S^\triangle(v)$ the linear system of triangle inequalities defined on the variables 
$w_a$, $w_b$, and~$w_c$ as follows:
\begin{eqnarray*}
  w_a & \leq & w_b + w_c \\
  w_b & \leq & w_a + w_c \\
  w_c & \leq & w_a + w_b \, .
\end{eqnarray*}
From $S^\triangle(v)$, one can derive that $w_a$, $w_b$, and $w_c$ are nonnegative. 
%
%
We denote by $S^{\calP}_t(v)$ the linear system of inequalities, at each internal node~$v$ 
of~$G$, resulting from imposing to~$S^\triangle(v)$ the additional \emph{perimeter inequality}
\begin{equation}\label{eq:t}
  w_a + w_b + w_c \ \leq \ t \, .
\end{equation}
Now, consider the union of all the linear systems~$S^{\calP}_1(v)$, 
taken over all internal nodes~$v$ of~$G$.  
Add the constraint $0 \leq w_e \leq 1/2$ 
for every edge $e$ in $G$ that alone is a component of~$G$.  
The set~$\calP_G$ consists of all real solutions 
for this linear system~\cite[Definition~2.3]{LiuO2006}.
These last two constraints imply that $\calP_G$ is a polytope, and we denote its Ehrhart quasi-polynomial by
\begin{equation}
L^{\calP}_G(t) := \volu_{\ZZ^d}(t\calP_G),
\end{equation}
where $d = |E(G)|$.

If~$G$ has connected components $H$ and $J$, 
then $\calP_G$ would be the Cartesian product of~$\calP_H$ and~$\calP_J$, 
and thus $L^{\calP}_G(t) = L^{\calP}_H(t) \cdot L^{\calP}_J(t)$ for all $t$.
Therefore, henceforth we assume that~$G$ is a connected graph.

In order to derive properties of $L^{\calP}_G$ for a $\{1,3\}$-graph $G$, Liu
and Osserman considered a polytope $\calQ_G$ in $\IR^{E}\times\IR^{I}$
closely related to $\calP_G$, where $I$ is the set of internal nodes of~$G$.  
To define $\calQ_G$, for each internal node $v$ of~$G$, 
with $a$,~$b$, and~$c$ being the three edges incident to~$v$,
consider an auxiliary variable~$z_v$ and let~$S^{\calQ}_t(v)$ be the previous linear system
$S^\triangle(v)$, together with the additional \emph{parity constraint}:
\begin{eqnarray}
  w_a + w_b + w_c & = & 2\,z_v  \label{eq:parity}\\
             z_v & \leq & t \, . \nonumber
\end{eqnarray}
We now consider the union of all the linear systems~$S^{\calQ}_1(v)$, 
taken over all internal nodes~$v$ of~$G$.
We add the constraint $0 \leq w_e \leq 1$ for every edge $e$ in $G$ that alone is a component of~$G$.
The polytope $\calQ_G$ consists of the solutions 
for this linear system~\cite[Definition~3.1]{LiuO2006}. 
For each nonnegative integer $t$, the number of integer points in the polytope $t\calQ_G$ 
is  denoted by $L^{\calQ}_G(t)$, the Ehrhart quasi-polynomial in $t$, associated to the polytope $\calQ_G$.
That is,
\begin{equation}
L^{\calQ}_G(t) := \volu_{\ZZ^d}(t\calQ_G),
\end{equation}
where $d = |E| + |I|$.  Figure~\ref{fig:involution} shows an example.

\begin{figure}[htb]
\begin{center}
\begin{minipage}[h]{0.5\textwidth}
  \scalebox{.9}{\includegraphics{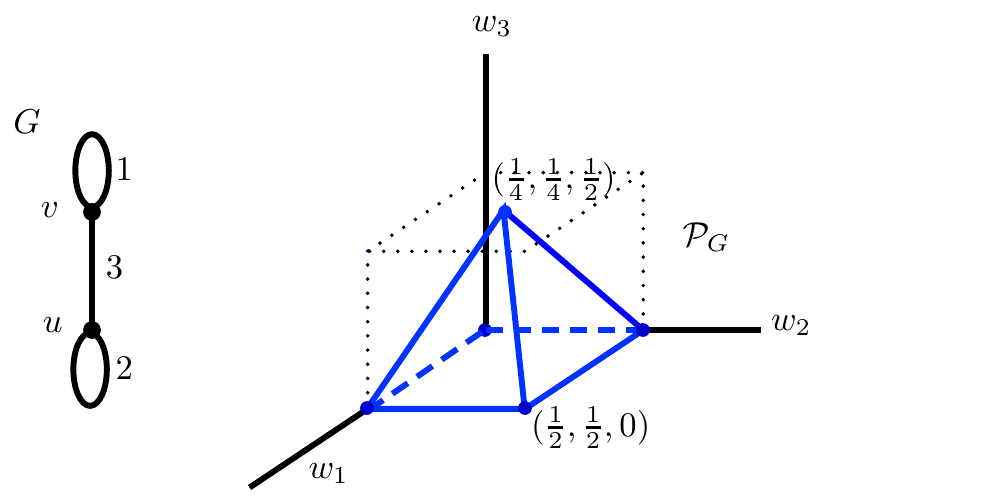}}
  \begin{align}
   L^{\calP}_G(t) & =  \frac{1}{24} t^3 + \frac{1}{4} t^2 + \,
    \left\{\begin{array}{ll}
             \frac{5}{6} t + 1,   & \mbox{if $t$ is even} \\[1mm] 
             \frac{11}{24} t + \frac{1}{4}, & \mbox{if $t$ is odd}
           \end{array}\right. \nonumber\\
   L^{\calQ}_G(t) & = \frac16 t^3 + t^2 + \frac{11}6 t + 1 \nonumber
\end{align}
\end{minipage}
\begin{minipage}[h]{0.4\textwidth}
{\small
  \[
    t\calP_G = \left\{ 
    \begin{array}{rcl}
      w_3 & \le & 2w_1 \\
      w_3 & \ge & 0 \\
      2w_1 + w_3 & \le & t \\
      w_3 & \le & 2w_2 \\
      2w_2 + w_3 & \le & t \\
    \end{array} \right.\\
\]
\[
  t\calQ_G = \left\{
    \begin{array}{rcl}
      w_3 & \le & 2w_1 \\
      w_3 & \ge & 0 \\
      2w_1 + w_3 & = & 2z_v \\
      z_v & \le & t \\
      w_3 & \le & 2w_2 \\
      2w_2 + w_3 & = & 2z_u \\
      z_u & \le & t
    \end{array} \right.
  \]
}
\end{minipage}
\end{center}
\caption{
  A cubic graph $G$, its polytope $\calP_G$, the linear systems of $t\calP_G$
  and $t\calQ_G$,
  and the Ehrhart quasi-polynomials $L^{\calP}_G(t)$ and~$L^{\calQ}_G(t)$~\cite{latte}.
  The polytope $\calQ_G$ lies in~$\IR^3 \times \IR^2$.
  The points in $t\calP_G$ have coordinates~$(w_1,w_2,w_3)$ and the points in $t\calQ_G$
  have coordinates~$((w_1,w_2,w_3),(z_v,z_u))$.}
\label{fig:involution}
\end{figure}

\begin{example}
For the graph $G$ in Figure~\ref{fig:involution},
$L^{\calP}_G(t)$ has period~2 and $L^{\calQ}_G(t)$ has period~1, therefore the number of
integer points in the polytope~$\calP_G$ is $L^{\calP}_G(1) =
\frac{1}{24} + \frac{1}{4} + \frac{11}{24} + \frac{1}{4} = 1$ and the
number of integer points in the polytope
$\calQ_G$ is $L^{\calQ}_G(1) = \frac16 + 1 + \frac{11}6 + 1 = 4$.
Indeed, the unique integer point in~$\calP_G$ is $(0,0,0)$ and the integer points in~$\calQ_G$ are
$((0, 0, 0), (0, 0))$, $((0, 1, 0), (0, 1))$, $((1, 0, 0), (1, 0))$, and $((1, 1, 0), (1, 1))$.
\hfill $\square$\\[1mm]
\end{example}

We denote the constituent polynomials of $L^{\calP}_G(t) $ by $p_0(t), p_1(t), p_2(t)$, and $p_3(t)$,
where 
\[
p_k(t) = L^{\calP}_G(t), \  \text{for } t \equiv k \Mod{4}\,.
\]
If  $L^{\calP}_G(t)$   has period~1, then $p_0 = p_1 = p_2 = p_3$.
If~$L^{\calP}_G(t)$ has period~2, then~$p_0 = p_2$ and $p_1 = p_3$.
Similarly, we denote by~$q_0(t)$ and~$q_1(t)$ the constituent polynomials of~$L^{\calQ}_G(t)$,
where 
\[
q_k(t) = L^{\calQ}_G(t),  \text{ for }  t \equiv k \Mod{2}\,.
\]
If~$L^{\calQ}_G(t)$ has period~1, then $q_0= q_1$.

\begin{example}
For the graph $G$ in Figure~\ref{fig:involution}, we have that 
$$
p_0(t) = p_2(t) = \frac{1}{24} t^3 + \frac{1}{4} t^2 + \frac{5}{6} t + 1 \mbox{ \ and \ }
p_1(t) = p_3(t) = \frac{1}{24} t^3 + \frac{1}{4} t^2 + \frac{11}{24} t + \frac{1}{4},
$$ 
so that here $L^{\calP}_G(t)$ has period $2$.
The quasi-polynomial $L^{\calQ}_G(t)$ has period~1 because 
$$q_0(t)= q_1(t)= \frac16 t^3 + t^2 + \frac{11}6 t + 1\;.$$
The vertices of $\calQ_G$ are its integer points and the point
$((\frac{1}{2}, \frac{1}{2}, 1), (1, 1))$.
The least common multiple of denominators in the coordinates 
of the vertices of~$\calP_G$  is~$4$ and of~$\calQ_G$ is~$2$, giving us examples of period collapse.
\hfill $\square$\\[1mm]
\end{example}

\subsection{Hightlighting context of our main results}

Liu and Osserman~\cite[Proposition~3.5]{LiuO2006} proved that, in general, the coordinates of all
vertices of~$\calP_G$ are in~$\{0,\frac14,\frac12,1\}$, while the coordinates
of all vertices of $\calQ_G$ are in $\{0,\frac12,1\}$.  Therefore, the period
of $L^{\calP}_G(t)$ is either~1 or~2 or~4 and the period of $L^{\calQ}_G(t)$ is either~1 or~2.
In particular, Liu and Osserman~\cite[Question~4.3]{LiuO2006} posed some questions 
related to the period of the Ehrhart quasi-polynomials $L^{\calP}_G(t)$ and $L^{\calQ}_G(t)$,
restated as follows.

\begin{question}[{\cite[Question~4.3]{LiuO2006}}]\label{qst:liu} {\ }
  \begin{enumerate}[{\normalfont (a)}] 
  \item Is it true that if $G$ is cubic then the period of 
    the Ehrhart quasi-polynomial $L^{\calP}_G(t)$ is~2?\label{qst:liu:a}
  \item For which $\{1,3\}$-graphs $G$ is the period of $L^{\calP}_G(t)$ smaller than 
    the least common multiple of the denominators of the vertices
    of~${\calP}_G$? \label{qst:liu:b}
  \item Is the period of the Ehrhart quasi-polynomial $L^{\calQ}_G(t)$ always half 
    the period of $L^{\calP}_G(t)$ for every $\{1,3\}$-graph~$G$? \label{qst:liu:c}
  \end{enumerate}  
\end{question}  

This paper gives a partial answer to some of these questions.  
We prove the following.

\begin{thm}[the period for $\{1,3\}$-trees]\label{thm:period2trees}
  If $T$ is a $\{1,3\}$-tree, 
  then $L^{\calP}_T(t)$ has period~2 and $L^{\calQ}_T(t)$ has period 1.
\end{thm}

Theorem \ref{thm:period2trees}  is related to 
Questions~\ref{qst:liu}\eqref{qst:liu:b} and~\ref{qst:liu}\eqref{qst:liu:c}.

\begin{thm}[the period for $\{1,3\}$-graphs]\label{thm:main}
If $G$ is a connected $\{1,3\}$-graph then the period of the Ehrhart quasi-polynomial
$L^{\calP}_G(t)$  associated to the polytope $\calP_G$
is at most~2 if~$G$ is a tree or a cubic graph, and it is equal to~$4$ otherwise.
\end{thm}

Theorem \ref{thm:main} answers Question~\ref{qst:liu}\eqref{qst:liu:a} 
and is one of the main contributions of this paper. 
Liu and Osserman~\cite[Lemma~3.3]{LiuO2006} proved that the constituent 
polynomials $p_1=p_3$ by showing that, for every $\{1,3\}$-graph $G$, 
$q_1(t) = N_G\,p_1(t) = N_G\,p_3(t)$, where~$N_G$ is the number of internally 
Eulerian subgraphs of $G$.

\begin{example}
For the graph $G$ in Figure~\ref{fig:involution}, the internally Eulerian
subgraphs are induced by the edge sets $\emptyset$, $\{1\}$, $\{2\}$, 
and $\{1,2\}$, so $N_G=4$, and we have
$$q_1(t) \ = \ \frac16 t^3 + t^2 + \frac{11}6 t + 1 
         \ = \ N_G\, p_1(t) 
         \ = \ 4\, \big(\frac{1}{24} t^3 + \frac{1}{4} t^2 + \frac{11}{24} t + \frac{1}{4}\big)\;. \\[-9mm]
         $$
\hfill $\square$\\[1mm]
\end{example}         

Their proof \cite[Lemma~3.3]{LiuO2006} is based on a partition, for every nonnegative odd integer~$t$,
of the~{$L^{\calQ}_G(t)=q_1(t)$} integer points of the polytope $t\calQ_G$
into~$N_G$ parts of size~$L^{\calP}_G(t)$.
So, answering Question~\ref{qst:liu}\eqref{qst:liu:a} 
boils down to deciding whether we have equality between the two constituent polynomials $p_0$ and  $p_2$ for cubic graphs. 

We answer Question~\ref{qst:liu}\eqref{qst:liu:a} positively by also presenting, 
for every cubic graph $G$ and nonnegative even integer $t$, 
a partition of the~$L^{\calQ}_G(t)=q_0(t)$ integer points of the polytope~$t\calQ_G$.
This partition has $N_G$ parts,
one part of size~$L^{\calP}_G(t)$
and~$N_G-1$ parts of
size~$L^{\calP}_G(t) - (\frac{t}{2}+1)^{k-1}$,
where~$m$ and $n$ are the number of edges and nodes of~$G$, respectively,
and $k=m-n+1$ is the \defi{cyclomatic number} of $G$.
From this it follows that 
$$q_0(t) \ = \ N_G\, p_0(t) - (N_G-1)\,\big(\frac{t}{2}+1\big)^{k-1} 
           = \ N_G\, p_2(t) - (N_G-1)\,\big(\frac{t}{2}+1\big)^{k-1}\,,$$
implying that $p_0 = p_2$ and that the period of~$L^{\calP}_G(t)$ is at most~2. 
This also shows that the class of cubic graphs is a class of $\{1,3\}$-graphs 
as sought-after in Question~\ref{qst:liu}\eqref{qst:liu:b}.

\begin{example}
Inspecting the example in Figure~\ref{fig:involution} a bit further,
we see that $k = 3-2+1 = 2$ and, because $N_G=4$, we have that
\begin{align*}
  q_0(t) & = \frac16 t^3 + t^2 + \frac{11}6 t + 1 \\
         & = N_G\, p_0(t) - (N_G-1)\,\big(\frac{t}{2}+1\big)^{k-1} \\
         & = 4\, \big(\frac{1}{24} t^3 + \frac{1}{4} t^2 + \frac{5}{6} t + 1\big) -
             3\, \big(\frac{t}{2}+1\big)\,.\\[-15mm]
\end{align*}
\hfill $\square$\\[1mm]
\end{example}
  
We were not able to show that, for a cubic graph $G$, 
the period of $L^{\calP}_G(t)$ is exactly~2, that is, that $p_0 \neq p_1$.  
However, we point out that the periods of $L^{\calP}_G(t)$ and $L^{\calQ}_G(t)$ are different.
Indeed, from the above, 
\begin{align}
q_0(1) & = N_G\,p_0(1) - (N_G{-}1)(\frac32)^{k-1} = 2^k\,p_0(1) - (2^k-1)(\frac32)^{k-1} \mbox{\ and \ } \nonumber \\
q_1(1) & = N_G\,p_1(1) = N_G = 2^k.\nonumber
\end{align}
Therefore, if ~$L^{\calQ}_G(t)$ has period~1, then $q_0 = q_1$, 
hence~$q_0(1) = q_1(1)$, which implies 
that~$p_0(1) = 1 + \frac{2^k-1}{2^k}(\frac32)^{k-1} \neq 1 = p_1(1)$, because $k \geq 2$, 
thus~$L^{\calP}_G(t)$ has period~2.
Similarly, if $L^{\calP}_G(t)$ has period~1, then $L^{\calQ}_G(t)$ has period~2.
This tackles Question~\ref{qst:liu}\eqref{qst:liu:c}.

\newpage

Finally, we derive that the period of~$L^{\calP}_G(t)$ is~4 for any $\{1,3\}$-graph~$G$ 
that is not a tree or a cubic graph, applying the same strategy used for cubic graphs.
This shows that cubic graphs are the only ones that satisfy 
Question~\ref{qst:liu}\eqref{qst:liu:b}. 
Together with Theorem~\ref{thm:period2trees}, this result leads to Theorem~\ref{thm:main}.

Although our approach allows us to control the behavior for the difference of polynomials, 
we were not able, despite many efforts, to find a method to compute the desired polynomials 
explictly. This seems a challenging task even for $\{1,3\}$-trees as stated 
in~\cite[Problem 6.5]{FernandesPRAR2020}.

\subsection{Paper organization}

The paper is organized as follows. 
In Section~\ref{sec:13trees}, we show that, for every $\{1,3\}$-tree~$T$, 
the coordinates of all vertices of $\calP_T$ are in $\{0,\frac12\}$
(Theorem~\ref{thm:thevertices}). 
We also derive that~$\calQ_T$ is a 0/1 polytope (Theorem~\ref{thm:thevertices2})
yielding Theorem~\ref{thm:period2trees}.
In Section~\ref{sec:nnis} we carry out the decomposition of 
the integer points of the polytope $\calQ_G$ that will lead to Theorem~\ref{thm:main}.
We do so by studying~$L^{\calQ}_G$ via cosets of certain lattices and 
based on internally Eulerian subgraphs of~$G$.
The latter allows us to establish a connection with $L^{\calP}_G$ 
through a simple class of representative trees called caterpillars.
To this end, we use the NNI machinery (an easy local move performed 
on the edges of~$G$). 
Section~\ref{sec:theperiod} presents the proof of Theorem~\ref{thm:main}
using two main technical lemmas that are proved in 
Sections~\ref{sec:leafs} and~\ref{sec:loops}.
Finally, in Section~\ref{sec:furtherconnexions}, 
we discuss further geometric properties of the polytopes $\calP_T$ 
and a related topological connection. 

\section{The structure and period for $\mathbf{\{1,3\}}$-trees}
\label{sec:13trees}

Let $T$ be a $\{1,3\}$-tree.
We say that $T$ is \emph{trivial} if it has no internal node.
A \emph{leaf-edge} is an edge that is incident to a leaf.
A path in $T$ is a \emph{leaf-path} if its extreme edges are leaf-edges.
An internally Eulerian subgraph in $T$ is simply a collection of disjoint leaf-paths.

If $x$ be a point in $\calP_T$, then $0 \leq x_e \leq \frac12$
for each edge $e$.
Suppose that all coordinates of $x$ are in $\{0,\frac12\}$.
If  $a$, $b$, and $c$ are edges incident to a node, then
either zero or two of~$x_a$, $x_b$, and~$x_c$
are equal to~$\frac12$.
This implies that the support of $x$ is a set of
edges of a collection of disjoint leaf-paths.
We shall prove that
${V_{\calP_T} := \{\frac12\um_H : \mbox{$H$ is a collection of disjoint leaf-paths in $T$}\}}$
is the set of vertices of $\calP_T$.

First we show that any point in $\calP_T$ with all coordinates 
in $\{0,\frac12\}$ is uniquely determined by the coordinates
associated to leaf-edges.

\begin{figure}[htb]  
\begin{center}
\begin{minipage}[h]{0.5\textwidth}
  \scalebox{.9}{\includegraphics{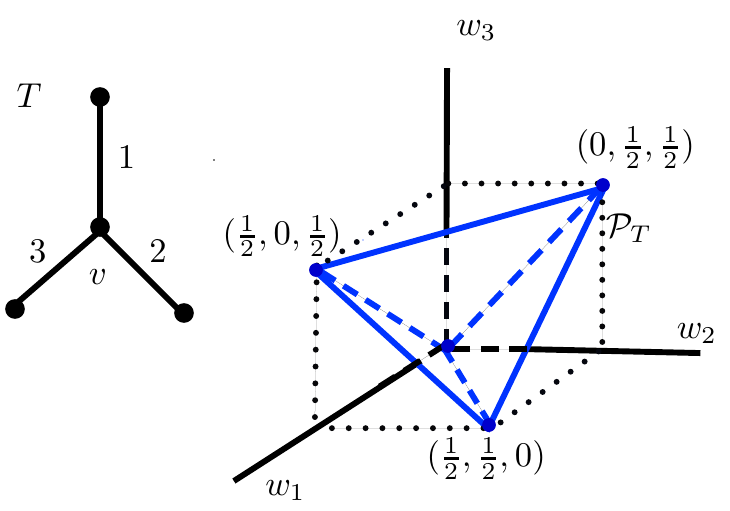}}
  \begin{align}
   L^{\calP}_T(t) & =  \frac{1}{24} t^3 + \frac{1}{4} t^2 + \,
    \left\{\begin{array}{ll}
             \frac{5}{6} t + 1,   & \mbox{if $t$ is even} \\[1mm] 
             \frac{11}{24} t + \frac{1}{4}, & \mbox{if $t$ is odd}
           \end{array}\right. \nonumber\\
   L^{\calQ}_T(t) & = \frac16 t^3 + t^2 + \frac{11}6 t + 1 \nonumber
\end{align}
\end{minipage}
\begin{minipage}[h]{0.4\textwidth}
{\small
  \[
    t\calP_T = \left\{ 
    \begin{array}{rcl}
      w_1 & \le & w_2 + w_3\\
      w_2 & \le & w_1 + w_3\\
      w_3 & \le & w_1 + w_2\\
      w_1 + w_2 + w_3 & \le & t 
    \end{array} \right.\\
\]
\[
  t\calQ_T = \left\{
  \begin{array}{rcl}
    w_1 & \le & w_2 + w_3\\
    w_2 & \le & w_1 + w_3\\
    w_3 & \le & w_1 + w_2\\
    w_1 + w_2 + w_3 & = & 2z_v \\
      z_v & \le & t     \end{array} \right.
  \]
}
\end{minipage}
\end{center}
\caption{\label{fig:involutionT}
  A $\{1,3\}$-tree $T$, its polytope $\calP_T$, the linear systems of $t\calP_T$
  and $t\calQ_T$, and
  the Ehrhart quasi-polynomials $L^{\calP}_T(t)$ and~$L^{\calQ}_T(t)$~\cite{latte}.
  The Ehrhart quasi-polynomials associated to $T$ are equal to the 
  Ehrhart quasi-polynomials associated to the cubic graph $G$ of Figure~\ref{fig:involution}.
  The points in $t\calP_G$ have coordinates~$(w_1,w_2,w_3)$ and the points in $t\calQ_G$
  have coordinates~$((w_1,w_2,w_3),(z_v))$.
} 

\end{figure}

\begin{example}
  For the $\{1,3\}$-tree $T$ in Figure~\ref{fig:involutionT} the set with all collections of
  disjoint leaf-paths is $\{ \emptyset, \{1,2\},\{2,3\},\{1,3\}\}$
  and $V_{\calP_T} = \{(0,0,0), (\frac12,\frac12,0), (0,\frac12,\frac12), (\frac12,0,\frac12)\}$
  is the vertex set of $\calP_T$.
\hfill $\square$
\end{example}  

\newpage

\begin{lem}\label{lem:PTuniqueness}
  If $T$ is a $\{1,3\}$-tree,
  $X_T$ is the set of leaf-edges of~$T$,
  and~$X$ is a subset of~$X_T$ with $|X|$ even,
  then there exists a unique point $x$ in~$\calP_T$ such that
  \begin{align}
    \label{eq:vertex}
  x_e & =\begin{cases}
  \frac12 & \text{if $e$ in $X$}\\
  0       & \text{if $e$ in $X_T \setminus X$}.
  \end{cases}
  \end{align}
  Moreover, all coordinates of $x$ are in $\{0,\frac12\}$ and
  its support is the set of edges of a collection of disjoint leaf-paths.
\end{lem}

\begin{proof}
  The proof is by induction on the number of nodes of $T$.
  If $T$ is trivial or has only one internal node, then all edges of $T$
  are leaf-edges and $x$ is completely defined by~\eqref{eq:vertex} and satisfies all conditions of the lemma. 
  Thus, we may assume that $T$ has at least two internal nodes. 
  
  Let $r$ be an internal node of $T$ adjacent to two leaf-edges,
  and let $a$, $b$, and~$c$ be the three edges incident to~$r$.
  We may assume that $b$ is incident to a leaf $u$ and $c$ is incident
  to a leaf $v$.
  Take $T' := T - \{u, v\}$ and
  $X' := X \setminus \{b, c\} \cup \{a\}$ if $|X \cap \{b, c\}| = 1$ and
  $X' := X \setminus \{b, c\}$ otherwise.
  By induction, there exists a unique point $x'$ in $\calP_{T'}$ with
  coordinates in $\{0,\frac12\}$ such that for each leaf-edge~$e'$ of $T'$
  we have that $x'_e = \frac12$ if and only if $e$ in $X'$ and such that
  the support of $x'$ is the set of edges of a collection of
  disjoint leaf-paths of $T'$.
  One can verify that~$x'$ can be uniquely extended to a point $x$ in
  $\calP_{T}$ with coordinates in $\{0,\frac12\}$ satisfying the
  conditions of the lemma.
\end{proof}

Next, we show that any point in $\calP_T$ with coordinates
in $\{0,\frac12\}$ whose support is a collection of disjoint
leaf-paths is a vertex of $\calP_T$.

\begin{lem}\label{lem:PTvertices}
  If $T$ is a $\{1,3\}$-tree and $H$ is a
  collection of disjoint leaf-paths in~$T$, 
  then~$\frac12\um_H$ is a vertex of $\calP_T$.
\end{lem}

\begin{proof}
  It is clear that $\frac12\um_H$ is in $\calP_T$.
  Let $X_T$ be the set of leaf-edges of $T$ and
  $X$ be the edges of~$H$ that are leaf-edges of $T$.
  The hyperplane $2\sum_{e \in X} w_e -2 \sum_{e \in X_T \setminus X} w_e = |X|$ is a 
  supporting hyperplane of~$\calP_T$.
  Indeed, for each point $x$ in~$\calP_T$ and each edge $e$, we have
  that $0 \leq x_e \leq \frac12$.
  Thus,
  $2\sum_{e \in X} x_e - 2 \sum_{e \in X_T \setminus X} x_e \leq 2 \sum_{e \in X} x_e  \leq |X|$
  and equality holds if and only if $x_e = 1/2$ for each $e$ in $X$ and $x_e = 0$
  for each~$e$ in $X_T \setminus X$.
  By Lemma~\ref{lem:PTuniqueness}, $\frac12\um_H$ is the unique point in $\calP_T$
  for which equality holds and therefore it is a vertex~of $\calP_T$.
\end{proof}

\begin{thm}\label{thm:thevertices}
  If $T$ is a $\{1,3\}$-tree, then $V_{\calP_T} $ is the set of vertices of the polytope~$\calP_T$.
\end{thm}

\begin{proof}
  Let $\calR_T$ be the set of points that are a convex combination of points in $V_{\calP_T} $.
  By Lemma~\ref{lem:PTvertices} we have that $\calR_T \subseteq \calP_T$.
  In order to prove the lemma it suffices to show the converse inclusion.

  Suppose that $\calP_T \not\subseteq \calR_T$ and let $T$ be the smallest $\{1,3\}$-tree
  with  $\calP_T \not\subseteq \calR_T$; that is, with the number of nodes as small as possible.
  It is clear that $T$ is nontrivial.
  Let $x$ be a vertex of~$\calP_T$ not contained in~$\calR_T$. 

  Since $x$ is a vertex of $\calP_T$, there are $|E|$ linearly independent
  inequalities among the ones that determine $\calP_T$
  satisfied by $x$ with equality.
  As $|E| = 2|I| + 1$ there must exist an internal node $v$ such
  that three inequalities in $S^{\calP}_1(v)$ are satisfied by $x$
  with equality.
  Let~$a, b$, and $c$ be the three edges incident to $v$.
  We have two possibilities:
  \begin{itemize}
  \item either the three inequalities
    in $S^{\triangle}(v)$ are satisfied by $x$ with equality and
    therefore $x_a = x_b = x_c = 0$;
  \item or the perimeter inequality and two inequalities in $S^{\triangle}(v)$
    are satisfied by~$x$ with equality and therefore two of $x_a$, $x_b$,
    and $x_c$ are $\frac12$ and the other is~$0$.
  \end{itemize}
  
  The tree~$T$ can be partitioned into $\{1,3\}$-trees
  $T_a$, $T_b$, and~$T_c$ that share the node $v$
  and have~$a$, $b$, and $c$ as leaf-edges, respectively. 
  Let $x^{}_{T_a}$, $x^{}_{T_b}$, and $x^{}_{T_c}$ be the corresponding `projections' of $x$
  onto the edges of $T_a$, $T_b$, and $T_c$, respectively.
  One can verify that in both possibilities ($x_a= 0$ or $x_a = \frac12$)
  $x^{}_{T_a}$ is in $\calP_{T_a}$.
  Since $T_a$ is smaller than~$T$, it follows that $x^{}_{T_a}$ is in $\calR_{T_a}$ and therefore
  $x^{}_{T_a}$ can be decomposed as a convex combination of points in~$V_{\calP_{T_a}}$.
  Similarly, we can decompose~$x^{}_{T_b}$ and~$x^{}_{T_c}$ as a convex combination of
  points in~$V_{\calP_{T_b}}$   and~$V_{\calP_{T_c}}$, respectively.
  These decompositions can be easily glued together to form a decomposition of $x$
  as a convex combination of points in $V_{\calP_T} $, contradicting our assumption.
\end{proof}

Using the same strategy, one can derive the following about $\calQ_T$.
Recalling that~$I$ is the set of internal nodes of~$G$,
let $\um'_H$ denote the pair $(\um_H,z)$ in $\ZZ^E \times \ZZ^{I}$, 
where~$z_v = 1$ if $v$ is an internal node of a leaf-path in $H$, and~$z_v = 0$ otherwise.
Let~${V_{\calQ_T} := \{\um'_H : \mbox{$H$ is a collection of disjoint leaf-paths in $T$}\}}$.
This discussion is summarized in the following theorem.

\begin{thm}\label{thm:thevertices2}
If $T$ is a $\{1,3\}$-tree, then $V_{\calQ_T}$ is the set of vertices of the polytope~$\calQ_T$.~\qed
\end{thm}

\begin{example}
  For the $\{1,3\}$-tree $T$ in Figure~\ref{fig:involutionT} we have that the vertex set of $\calQ_T$ is
  $V_{\calQ_T} = \{((0,0,0),(0)), ((1,1,0),(1)), ((0,1,1),(1)),((1,0,1),(1))\}$.
\hfill $\square$
\end{example}  

Now we can give a proof for Theorem \ref{thm:period2trees}, which  characterizes the periods of the Ehrhart quasi-polynomials for $\{1,3\}$-trees.

\begin{proof}[Proof of Theorem~\ref{thm:period2trees} (the period for $\{1,3\}$-trees)]
By Theorem~\ref{thm:thevertices2}, we known that $\calQ_T$ is integral.
It follows from Ehrhart's theorem~\cite[Theorem~3.8]{BeckS2009} that~$L^{\calQ}_T(t)$ has period 1
and if $q_0$ and $q_1$ are its constituents  then $q_0 = q_1$ and $q_1(0)=1$.
By Theorem~\ref{thm:thevertices}, we have that $\calP_T$ is half-integral and
$L^{\calP}_T(t)$ has period at most~2.
Let $p_0$ and~$p_1$ be the constituents of~$L^{\calP}_T(t)$.
It is clear that $p_0(0) = 1$ and the constant coefficient of $p_0$ is~1.
Liu and Osserman~\cite[Proposition~3.5]{LiuO2006} proved that $q_1 = N_T\,p_1$, where $N_T$
is the number of internally Eulerian subgraphs of $T$.
Therefore, the constant coefficient of $p_1$ is $1/N_T$ and, because  $N_T \geq 2$ for each $\{1,3\}$-tree, then $p_0 \neq p_1$.
Hence, the period of $L^{\calP}_T(t)$ is exactly~2.
\end{proof}





\section{Cosets and combinatorial tools}\label{sec:nnis}

The intent of this section is to derive a relation between
$L^{\calQ}_G(t)$ and~$L^{\calP}_G(t)$ for (even or odd) nonnegative integer $t$.
For this task, internally Eulerian subgraphs, a class of graphs that resembles
caterpillars, and a local move performed on graphs called nearest neighbor
interchange (NNI) play central roles.
The general idea is partitioning the integer points of the
polytope $t\calQ$ into parts of size having~$L^{\calP}_G(t)$ as our unit of
standard measurement.

\subsection{Eulerianicity}

Let $G$ be a $\{1,3\}$-graph and $\calL_G := (2\ZZ^E) \times \ZZ^I$. 
For each integer $t \geq 0$ and each internally Eulerian subgraph~$H$ of~$G$, 
we define 
\begin{equation}
  \label{eq:partevol}
\parte(G,H) \ := \ t\calQ_G \cap (\calL_G+(\um_H,0)) \mbox{\quad and \quad} 
  \vol(G,H) \ := \ |\parte(G,H)|, 
\end{equation}
where $\calL_G+(\um_H,0)$ is a coset of the lattice $\calL_G$, in the lattice 
$\ZZ^E \times \ZZ^I$.

In words, we are now counting integer points in cosets of a lattice, not
just in a lattice.
The term ``$\parte$'' originates from the fact that these sets
correspond to parts of a 
partition of the set of integer points of the polytope $t\calQ_G$.
We note that 
\begin{equation}
  \label{eq:partition}
  t\calQ_G \cap (\ZZ^E \times \ZZ^I) \ = \ \bigcup_H \parte(G,H)   
  \mbox{\quad and \quad}
  L^{\calQ}_G(t) \ = \ \sum_H \vol(G,H), 
\end{equation}
where the union and the summation are over all internally Eulerian subgraphs $H$ of $G$.
Indeed, because of the parity constraint~\eqref{eq:parity}, 
for every integer point $(w,z)$ in $t\calQ_G$, 
the set of edges~$e$ for which $w_e$ is odd induces an internally Eulerian subgraph of $G$.

\newpage

\begin{example}\label{example}
  Consider the polytope $t\calQ_G$ in Figure~\ref{fig:involution}.
  The graph~$G$ has four internally Eulerian subgraphs, namely, the subgraphs
  $\emptyset, H_{\{1\}}, H_{\{2\}}$, and $H_{\{1,2\}}$ induced by the
  edge sets $\emptyset$, $\{1\}$, $\{2\}$, and~$\{1,2\}$, respectively.
  One can verify that
\[
\begin{array}{ll}
  \parteu_2(G,\emptyset) & = \{((0,0,0), (0,0)), ((2,0,0),(2,0)),
  ((0,2,0),(0,2)), ((2,2,0),(2,2))\}, \nonumber \\  
\parteu_2(G,H_{\{1\}}) & = \{ ((1,0,0), (1,0)), ((1,2,0),(1,2)) \}, \nonumber \\  
\parteu_2(G,H_{\{2\}}) & = \{((0,1,0), (0,1)), ((2,1,0),(2,1))\}, \mbox{ and } \nonumber \\  
\parteu_2(G,H_{\{1,2\}}) & = \{((1,1,0), (1,1)), ((1,1,2),(2,2))\}. \nonumber
\end{array}
\]
Therefore, $L^{\calQ}_G(2) = 10 = \volu_2(G,\emptyset) + \volu_2(G,H_{\{1\}}) + \volu_2(G,H_{\{2\}}) + \volu_2(G,H_{\{1,2\}}) = 4+2+2+2.$
\hfill$\square$
\end{example}

The number of  internally Eulerian subgraph of $G$ is denoted by $N_G$.
If $G$ has $n$ nodes and $m$ edges, then its \defi{cyclomatic number} is $m-n+1$.
Liu and Osserman~\cite[Remark~3.11]{LiuO2006} observed that
if $G$ has $h$ leaf-edges and cyclomatic number $k$, then
$N_{G}=2^k$ if $h = 0$ and  $N_{G}=2^{k+h-1}$ if $h > 0$.

The remaining of this entire section is devoted to rewriting the summation
in~\eqref{eq:partition} in a way that its terms are more amenable
to be measured having~$L^{\calP}_G(t)$ as standard.
An important special term is $\vol(G,\emptyset)$.
This term is associated to
the set of integer points $\parte(G,\emptyset)$ formed by all integer points~$(w, z)$ 
in $t\calQ_G$ for which all coordinates of~$w$ are even.  Equivalently,
$$\parte(G,\emptyset) = \{(w, z) \in t\calQ_G: 
\mbox{$w = 2x$ for an integer point $x$ in $t\calP_G$}\}$$ 
since $z$ is uniquely determined by $w$ in order for~$(w,z)$ to be in $t\calQ_G$.
Therefore
\begin{equation}
  \label{eq:emptycoset}
  \vol(G,\emptyset) \ = \ |\parte(G,\emptyset)| \ = \ L^{\calP}_G(t).
\end{equation} 
For instance, for the graph $G$ appearing
in Figure~\ref{fig:involution}, $\volu_2(G,\emptyset) = 4 = L^{\calP}_G(2)$.

\subsection{Caterpillars}


Wakabayashi~\cite[Theorem~A(ii)]{Wakabayashi2013} proved that polytopes $\calP_G$ associated to 
any connected $\{1,3\}$-graph $G$ with a given number $n$ of nodes and a given number $m$ of edges
have the same  Ehrhart quasi-polynomial $L^{\calP}_G(t)$.
This gives us the freedom to elect a convenient representative for every equivalence
class of connected graphs with $n$ nodes and $m$ edges.
It turns out that, in order to derive the period of~$L^{\calP}_G(t)$,
a particularly convenient choice of a representative graph for~$G$
resembles a caterpillar.
We define this convenient class of representative graphs in the sequel.

A \emph{caterpillar} is a tree for which the removal of all leaves 
results in a path, called its \emph{central path}, or in the empty graph
(Figure~\ref{fig:Ghk}(a)).  
When the result is the empty graph, we set the central path to be the empty path. 
The edges that are not in the central path of the caterpillar are its \emph{legs}. 
The nodes that are not in the central path and are adjacent to an end node of the
  central path are its \emph{roots}.
When the central path is empty, both nodes are roots.
The legs adjacent to the roots are the \emph{stems} of the caterpillar.

For nonnegative integers $h$ and $k$ such that $h+k \geq 2$, let $T$ be a
$\{1,3\}$-tree which is a caterpillar with $h+k$ leaves.  A graph obtained
from $T$ by adding a loop to $k$ of its leaves is called an
\emph{$(h,k)$-caterpillar} (Figure~\ref{fig:Ghk}(b)).
A loop of an $(h,k)$-caterpillar is \emph{extremal} if its node is adjacent to
an end node of the central path.

The number of nodes in $T$ and in an
$(h,k)$-caterpillar obtained from $T$ is the same, that is, $2(h+k-1)$.

We denoted by~$G_{h,k}$ the $(h,k)$-caterpillar whose legs with loops are all
consecutive and legs without loops are also all consecutive
(Figure~\ref{fig:Ghk}(c)).  Due to the discussion above, henceforth all of our
ambient graphs will be $(h,k)$-caterpillars $G_{h,k}$.  To make the proofs
that follow easier to read, we use the notation
\[ 
  L^{\calP}_{h,k}(t) := L^{\calP}_{G_{h,k}}(t) \mbox{\quad and \quad} 
   L^{\calQ}_{h,k}(t) := L^{\calQ}_{G_{h,k}}(t).
\]

\begin{figure}[ht]
  \centering \scalebox{.75}{\includegraphics{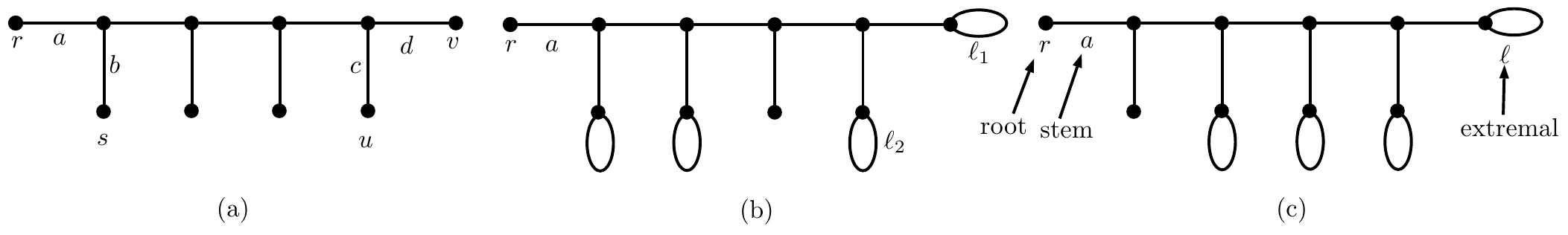}}
  \caption{(a) Caterpillar with 6 legs, with roots $r, s, u, v$  and stems $a, b, c, d$.
    (b) $(2,4)$-caterpillar with root $r$, stem $a$, and extremal loops $\ell_1$ and $\ell_2$.
    (c) $(2,4)$-caterpillar~$G_{2,4}$ with root $r$, stem $a$, and extremal loop $\ell$.}\label{fig:Ghk}
\end{figure}

\subsection{Nearest neighbor interchanges}

This section is devoted to simplify the summation in~\eqref{eq:partition}.
In the following two paragraphs we outline how we deal with the left and 
right side of the summation respectively. 

The cyclomatic number of $G_{h,k}$ is $k$.
From Wakabayashi~\cite[Theorem~A(ii)]{Wakabayashi2013} we know that if $G$ is a connected
$\{1,3\}$-graph with $h$ leaf-edges and cyclomatic number $k$,
then $L^{\calP}_G=L^{\calP}_{h,k}$.
The foremost step is to establish that $L^{\calQ}_G=L^{\calQ}_{h,k}$.

An internally Eulerian subgraph $H$ of any $(h,k)$-caterpillar
consists of a disjoint collection of leaf-paths and loops.
The next step is to prove that the value of
$\vol(G_{h,k},H)$ depends only on the number of leaf-paths and loops in~$H$.
Finally, we show that $\vol(G_{h,k},H)$ actually depends
on the number of leaf-paths in~$H$ and on whether or not $H$ has a loop.
For all these tasks we shall need the machinery of the nearest neighbor interchange.


A \emph{nearest neighbor interchange (NNI)} is a local invertible move performed 
in $G$ on a trail of length three, marked by dark edges in Figure~\ref{fig:nni}.  
This move interchanges one end of the two extreme edges of the trail on the central edge.

\begin{figure}[htbp]
  \centering
  \includegraphics{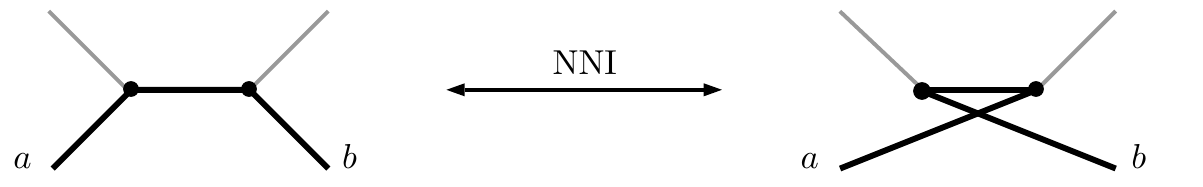}
  \caption{An NNI move on the trail marked by bold darker line segments.
    One of the incidences of the edges~$a$ and $b$ were interchanged.}
  \label{fig:nni}
\end{figure}

We refer to the central edge of the trail as the \emph{pivot} of the NNI move.
The result of the move is another $\{1,3\}$-graph~$G'$ on the same number of 
nodes, edges, and connected components. We call $\{G,G'\}$ an \emph{NNI pair}.

Consider an NNI pair $\{G,G'\}$.
Let $w$ and $z$ be weight functions defined on the edges and internal nodes of~$G$, respectively.
A \emph{weighted NNI} is a local invertible move performed on~$(G,w,z)$.
It is induced by the NNI pair.
The result of the move is~$(G', w', z')$ where~$w'$ and~$z'$ are 
weight functions on the edges and internal nodes of~$G'$ defined as follows.
Suppose that~$e=uv$ is the pivot of the NNI move.
Let~$a$ and~$b$ be the first and last edges in the trail,
and~$c$ and~$d$ be the remaining edges adjacent to~$e$,
with their corresponding weights depicted in
Figure~\ref{fig:weighted-nni}.
Note that $a$, $b$, $c$, and $d$ are not necessarily pairwise distinct.
The weight function $w'$, defined as in~\cite{FernandesPRAR2020},
is such that $w'_f = w_f$ for every~$f \neq e$ and
$$ w'_e = w_e + \max\{w_a+w_c,w_b+w_d\} - \max\{w_b+w_c,w_a+w_d\}. $$
The weight function $z$ is such that $z'_x = z_x$ for every $x \not\in \{u,v\}$,
$$z'_u = (w'_b+w'_c+w'_e)/2, \quad \mbox{and} \quad z'_v = (w'_a+w'_d+w'_e)/2.$$

\begin{figure}[htbp]
  \centering
  \includegraphics{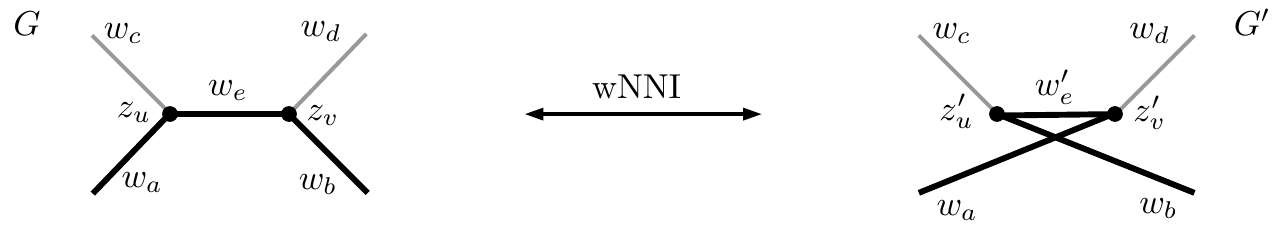}
  \caption{Weighted NNI move on the trail marked by bold darker line segments.
           Weights modified are the ones on the pivot edge and its end nodes.}
  \label{fig:weighted-nni}
\end{figure}

Let us restrict our attention to integer points 
satisfying the parity constraint~\eqref{eq:parity}.
If~$(w,z)$ is one of these points, then so is~$(w',z')$.
In fact, for any internally Eulerian subgraph $H$ of $G$, 
a weighted NNI move maps points in $\parte(G,H)$ into $\parte(G',H')$,
where $H'$ is uniquely determined by $\{G, G'\}$ and $H$.
Indeed, a weighted NNI acts on the parity (of the weight) of the NNI pivot based uniquely
on the parity of the edges incident to it (Figure~\ref{fig:parity-weighted-nni}).
Therefore, as a weighted NNI move is invertible, there is a bijection 
between~$\parte(G,H)$ and~$\parte(G',H')$ for every corresponding pair $H$ and $H'$.
We call $\{(G,H),(G',H')\}$ a \emph{weighted NNI pair}.

\begin{figure}[bhtp]
  \centering
  \scalebox{1.1}{\includegraphics{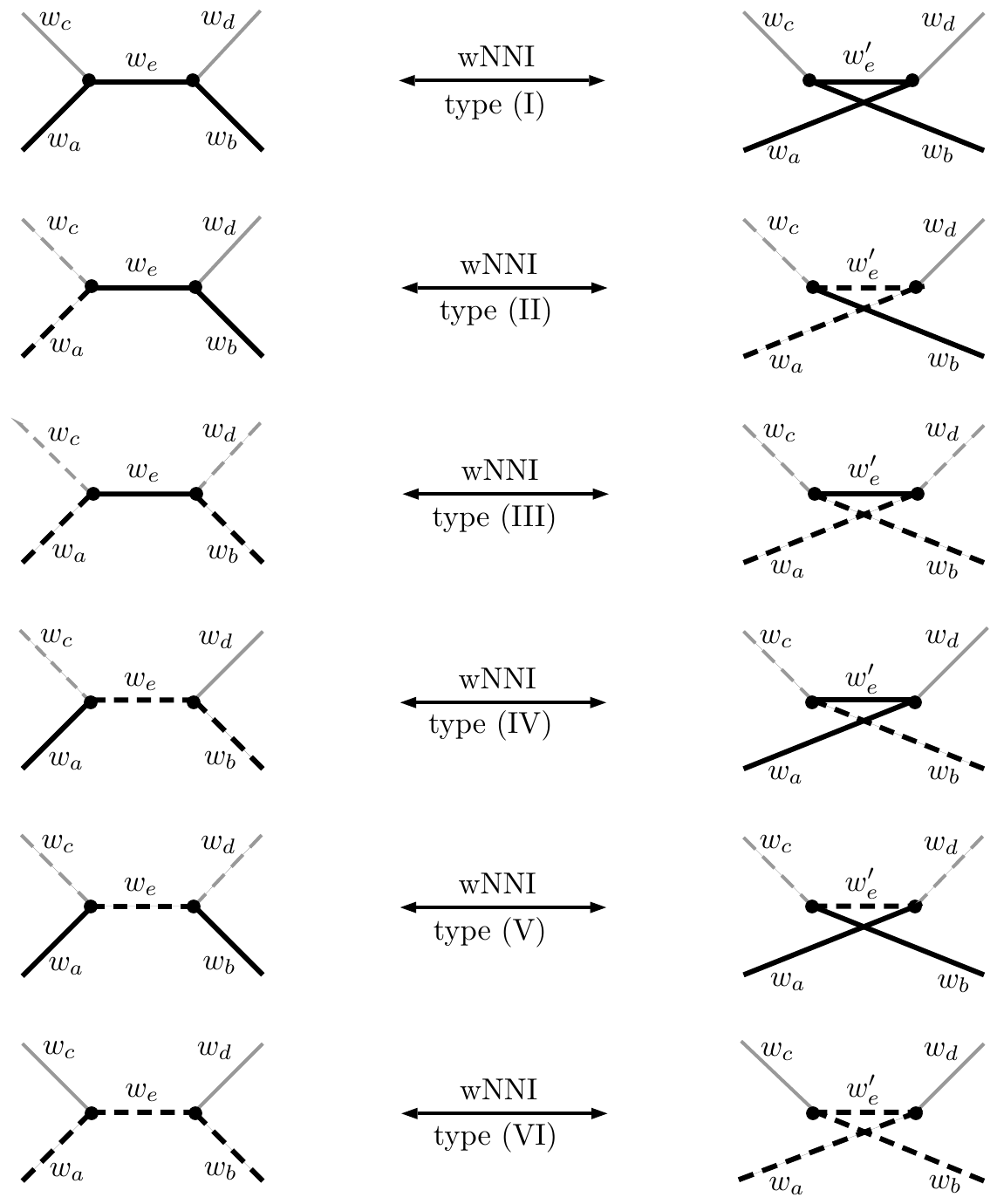}}
  \caption{Parity of weights on edges modified by a weighted NNI on a
    trail marked by bold darker line segments.
    Dashed line segments indicate edges with odd weights, and
    continuous line segments indicate edges with even weights.
    }
  \label{fig:parity-weighted-nni}
\end{figure}

The following lemma that summarizes the previous discussion is 
applied extensively, explicitly and implicitly, in several proofs ahead. 
\begin{lem}\label{claim:NNI}
  Let $\{(G,H),(G',H')\}$ be a weighted NNI pair.
  The corresponding weighted NNI is a bijection between 
  $\parte(G,H)$ and $\parte(G',H')$ and therefore 
  $$\vol(G,H) = \vol(G',H').$$ 
\end{lem}
\begin{figure}[htb]  
\begin{center}
\begin{minipage}[h]{0.5\textwidth}
  \scalebox{.9}{\includegraphics{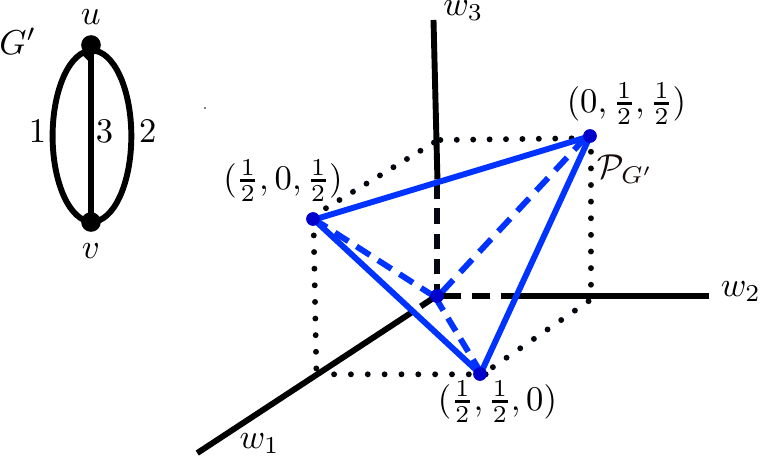}}
  \begin{align}
   L^{\calP}_{G'}(t) & =  \frac{1}{24} t^3 + \frac{1}{4} t^2 + \,
    \left\{\begin{array}{ll}
             \frac{5}{6} t + 1,   & \mbox{if $t$ is even} \\[1mm] 
             \frac{11}{24} t + \frac{1}{4}, & \mbox{if $t$ is odd}
           \end{array}\right. \nonumber\\
   L^{\calQ}_{G'}(t) & = \frac16 t^3 + t^2 + \frac{11}6 t + 1 \nonumber
\end{align}
\end{minipage}
\begin{minipage}[h]{0.4\textwidth}
{\small
  \[
    t\calP_{G'} = \left\{ 
    \begin{array}{rcl}
      w_1 & \le & w_2 + w_3\\
      w_2 & \le & w_1 + w_3\\
      w_3 & \le & w_1 + w_2\\
      w_1 + w_2 + w_3 & \le & t 
    \end{array} \right.\\
\]
\[
  t\calQ_{G'} = \left\{
  \begin{array}{rcl}
    w_1 & \le & w_2 + w_3\\
    w_2 & \le & w_1 + w_3\\
    w_3 & \le & w_1 + w_2\\
    w_1 + w_2 + w_3 & = & 2z_v \\
    w_1 + w_2 + w_3 & = & 2z_u \\
      z_v & \le & t \\ 
      z_u & \le & t     \end{array} \right.
  \]
}
\end{minipage}
\end{center}
\caption{
  The cubic graph $G'$ is obtained from applying an NNI to the graph $G$ in Figure~\ref{fig:involution}.
  The pivot of the NNI was the edge~3.
  The polytope $\calP_{G'}$, the linear systems of $t\calP_{G'}$
  and $t\calQ_{G'}$, and the Ehrhart quasi-polynomials associated to $G'$ are
  identical to the ones associated to the $\{1,3\}$-tree $T$
  in Figure~\ref{fig:involutionT}.
  The points in $t\calP_{G'}$ have coordinates~$(w_1,w_2,w_3)$ and the points in $t\calQ_{G'}$
  have coordinates~$((w_1,w_2,w_3),(z_v,z_u))$.
} 
\label{fig:involutionP}
\end{figure}
  
\begin{example}
  The graph $G$ in Figure~\ref{fig:involution} and $G'$ in Figure~\ref{fig:involutionP} form an
  NNI pair $\{G, G'\}$.
  Consider the polytope $t\calQ_{G'}$ in Figure~\ref{fig:involutionP}.
  The graph~$G'$ has four internally Eulerian subgraphs, namely, the subgraphs
  $\emptyset, H'_{\{1,2\}}, H'_{\{1,3\}}$  and  $H'_{\{2,3\}}$ induced by the
  edge sets $\emptyset$, $\{1,2\}$, $\{2,3\}$, and~$\{1,3\}$, respectively.
  One can verify that
\[
\begin{array}{ll}
\parteu_2(G',\emptyset) & = \{((0,0,0), (0,0)), ((2,2,0), (2,2)), ((2,0,2), (2,2)), ((0,2,2), (2,2)) \}, \nonumber \\  
\parteu_2(G',H'_{\{1,2\}}) & = \{ ((1,1,0), (1,1)), ((1,1,2), (2,2)) \}, \nonumber \\  
\parteu_2(G',H'_{\{1,3\}}) & = \{ ((1,0,1), (1,1)), ((1,2,1), (2,2)) \}, \mbox{ and } \nonumber \\  
\parteu_2(G',H'_{\{2,3\}}) & = \{ ((0,1,1), (1,1)), ((2,1,1), (2,2)) \}. \nonumber
\end{array}
\]
The corresponding weighted NNI pairs are
$\{ (G,\emptyset),  (G',\emptyset) \}$,
$\{ (G, H_{\{1\}}),   (G',H'_{\{1,3\}}) \}$,
$\{ (G, H_{\{2\}}),   (G',H'_{\{2,3\}}) \}$,
$\{ (G, H_{\{1,2\}}), (G',H'_{\{1,2\}}) \}$.
Therefore,
$\vol(G,H) = \vol(G',H')$ for each  weighted NNI pair $\{(G,H),(G',H')\}$,
confirming the previous Lemma \ref{claim:NNI}.
\hfill$\square$
\end{example}


Having the tool of weighted NNIs in hand, we are ready to follow the steps delineated.

\begin{lem}\label{lem:samepolynomials} 
  If $G$ is a connected $\{1,3\}$-graph with $h$ leaf-edges and cyclomatic number $k$ then
  $L^{\calQ}_G(t) = L^{\calQ}_{h,k}(t)$ for each nonnegative integer $t$.
\end{lem}

\begin{proof}
  The graph $G$ can be transformed into $G_{h,k}$ through a series of NNI moves~\cite[Theorem~1]{FernandesPRAR2020}.
  We have that \newpage
  \begin{align}
    L^{\calQ}_G(t) & =  \sum_H \vol(G,H) \label{eq:a} \\
    & =  \sum_H \vol(G_{h,k},H) \label{eq:b} \\
    & = L^{\calQ}_{h,k}(t),  \label{eq:c}
  \end{align}
  where the summations in~\eqref{eq:a} and in~\eqref{eq:b} are over
  the internally Eulerian subgraphs $H$ of~$G$ and~$G_{h,k}$,
  respectively.
  Equalities~\eqref{eq:a} and~\eqref{eq:c} are due to the summation in~\eqref{eq:partition},
  and   Equality~\eqref{eq:b} follows from Lemma  \ref{claim:NNI}    
   by induction on the number of NNI moves.
\end{proof}

Now, we prove that $\vol(G_{h,k},H)$ depends only on the number of 
leaf-paths and loops in $H$.

\begin{lem}\label{lem:sametypecosets}
  If $H$ and $H'$ are internally Eulerian subgraphs of~$G_{h,k}$ 
  with the same number of leaf-paths and loops, 
  then $$\vol(G_{h,k},H) = \vol(G_{h,k},H'),$$
  for each nonnegative integer $t$.
\end{lem}

\begin{proof}
  Consider a series of weighted NNIs taking the loops in~$H$ 
  to the ones in~$H'$, by swapping consecutive legs of $G_{h,k}$ and an 
  appropriated relabeling.  
  Similarly, consider a series of weighted NNIs taking the 
  leaf-edges of leaf-paths in~$H$ to the ones in~$H'$, by swapping 
  consecutive legs of $G_{h,k}$ and an appropriated relabeling.  
  The composition of the corresponding bijections is 
  a bijection between $\parte(G_{h,k},H)$ and $\parte(G_{h,k},H')$.
  Indeed, the used NNIs are as in Corollary~10 from~\cite{FernandesPRAR2020}, 
  and preserve the parity of the values on the corresponding leaf-edges and loops.
  So the lemma follows. 
\end{proof}

Any internally Eulerian subgraph of $G_{h,k}$ consists of 
a disjoint collection of $i$ leaf-paths and $j$ loops.  
Inspired by Lemma~\ref{lem:sametypecosets}, we shall start to use the following notation:
$H_{i,j}$  denotes any internally Eulerian subgraph with $i$ leaf-paths and $j$ loops.  
Recalling~\eqref{eq:partevol} in terms of the present notation, we have 
\begin{align}
 \parte(G_{h,k},H_{i,j}) & := t\calQ_{G_{h,k} }\cap (\calL_{G_{h,k}} + (\um_{H_{i,j}},0)), \mbox{ and} \nonumber \\
 \vol(G_{h,k},H_{i,j})  & := |\parte(G_{h,k},H_{i,j})|. \nonumber
\end{align}
Thus we may rewrite the summation in~\eqref{eq:partition}, namely that 
  $L^{\calQ}_G(t) \ = \ \sum_{H}   \vol(G,H)$,  to obtain 
\begin{equation}
  \label{eq:partition2}
L^{\calQ}_{h,k}(t) \ = \ \sum_{H_{i,j}}  
       \vol(G_{h,k},H_{i,j}) \ = \ \sum_{i=0}^{\floor{h/2}} {h \choose 2i} \sum_{j=0}^k{k \choose j}\vol(G_{h,k},H_{i,j}),
\end{equation}
where the first summation is over all internally Eulerian subgraphs $H_{i,j}$ of $G_{h,k}$.

Next we strengthen Lemma~\ref{lem:sametypecosets} and show that $\vol(G_{h,k},H)$ 
does not depend on the exact number of loops in~$H$, but only on whether $H$ has a loop or not. 
In other words, the value of $\vol(G_{h,k},H_{i,j})$ depends only on whether $j=0$ or $j \geq 1$.
This will allow us to simplify the summation in~\eqref{eq:partition2}.

\begin{lem}\label{lem:onelooptype}
  Let $h$ and $k$ be nonnegative integers such that $k \geq 1$ and $h+k \geq 2$.
  For every nonnegative integer $t$, $i = 0,\ldots,\lfloor h/2 \rfloor$, and $j=1,\ldots,k,$
  we have that $$\vol(G_{h,k},H_{i,j})= \vol(G_{h,k},H_{i,1}).$$
\end{lem}

\begin{proof}
  If $h=i=0$ and $k=j=2$, then the composition of the corresponding two weighted NNIs
  shown in Figure~\ref{fig:decreaseloops}(a) and a relabeling is by Lemma~\ref{claim:NNI}
  a bijection between $\parte(G_{0,2},H_{0,2})$ and $\parte(G_{0,2},H_{0,1})$,
  where the Eulerian subgraphs~$H_{0,2}$ and~$H_{0,1}$ are
  induced by the edge sets $\{a,c\}$ and $\{c\}$, respectively.

  For the other cases, the proof is by induction on $j$.
  Suppose that $j \geq 2$ and let $H$ be an internally Eulerian subgraph of $G_{h,k}$ 
  with $i$ leaf-paths and $j$ loops. 
  By applying weighted NNIs, 
  we may assume that two of the loops in $H$ are at distance~2.
  Call $a$ one of these loops.
  The composition of the corresponding  four weighted NNIs shown in Figure~\ref{fig:decreaseloops}(b) 
  and a relabeling is by Lemma~\ref{claim:NNI} a bijection between $\parte(G_{h,k},H)$ and $\parte(G_{h,k},H{-}a)$.
  Therefore, $\vol(G_{h,k},H_{i,j}) = \vol(G_{h,k},H_{i,j-1}) = \vol(G_{h,k},H_{i,1}),$
  where the second equality follows from the induction hypothesis.
\end{proof}

\begin{figure}[ht]
  \centering 
  \includegraphics[width=.9\textwidth]{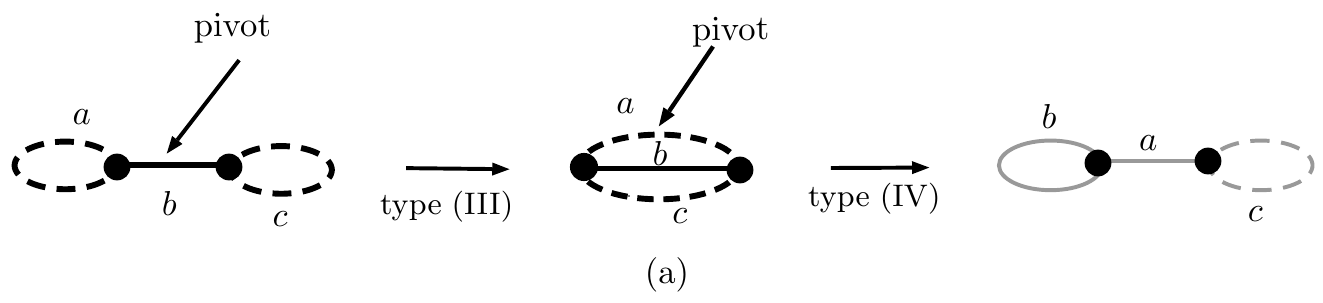}

  \vspace{5mm}

  \includegraphics[width=.9\textwidth]{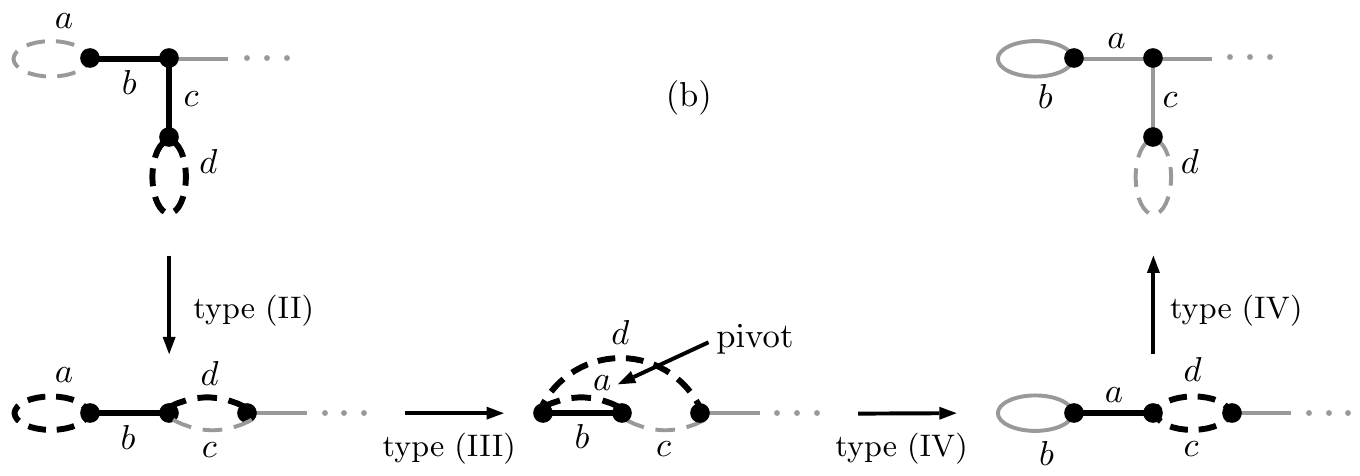}
  \caption{Weighted NNIs that bring two loops with odd weights to one loop with odd weight.
    The trail of each NNI is marked by bold darker line segments.
    Dashed line segments indicate the edges with odd weights.}
  \label{fig:decreaseloops}
\end{figure}

Finally, we take special attention to the value of $\vol(G_{h,k},H_{i,0})$ for $i > 0$.
This will allow us to simplify even further the summation in~\eqref{eq:partition2}.

\begin{lem}\label{lem:somepathzeroloops}
  Let $h$ and $k$ be nonnegative integers such that $k \geq 1$ and $h+k \geq 2$.
  For every nonnegative integer $t$, $i = 1,\ldots,\lfloor h/2 \rfloor$,
  we have that $$\vol(G_{h,k},H_{i,0})= \vol(G_{h,k},H_{i,1}).$$
\end{lem}

\begin{proof}
  Let $H$ be an internally Eulerian subgraph of $G_{h,k}$ 
  with $i$ leaf-paths and no loops.
  By applying weighted NNIs we may assume that there is a leaf-path of $H$ at distance~1 of a loop of $G_{h,k}$.
  We denote the loop by $a$. 
  The composition of the corresponding bijections of
  the two weighted NNIs ilustrated in Figure~\ref{fig:increaseloop} 
  and a relabeling is a bijection between $\parte(G_{h,k},H)$ and $\parte(G_{h,k},H{+}a)$.
  Therefore, $\vol(G_{h,k},H_{i,0}) = \vol(G_{h,k},H_{i,1}).$
\end{proof}

\begin{figure}[ht]
  \centering
  \includegraphics{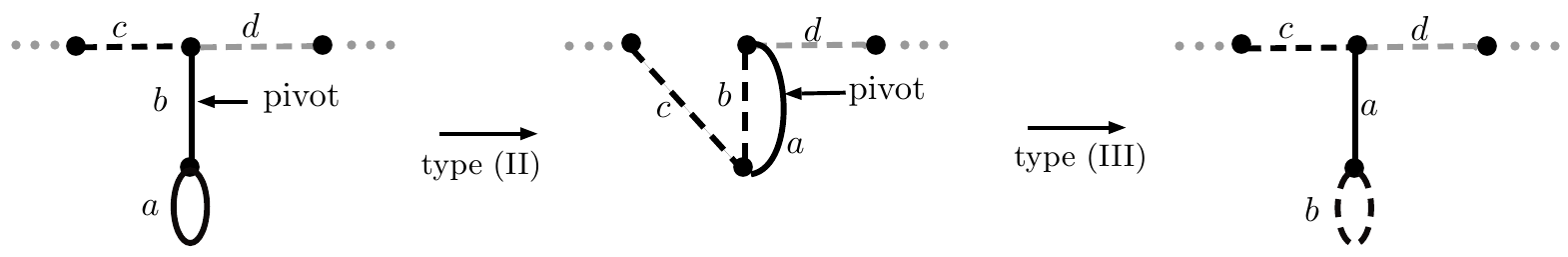}
  \caption{Weighted NNIs that, using a leaf-path with odd weights, change the weight of a loop from even to odd.
    The trail of each NNI is marked by bold darker line segments.
    Dashed line segments indicate the edges with odd weights.}
  \label{fig:increaseloop}
\end{figure}

\subsection{Key summation}
       
Finally, we arrive at the targeted expression of this section by       
manipulating the summation in~\eqref{eq:partition2}.
If $G$ is a $\{1,3\}$-graph with $h$ leaf-edges and cyclomatic number $k$ then
\begin{align}
  L^{\calQ}_G(t) & =  L^{\calQ}_{h,k}(t) \label{eq:waka} \\ 
               & = \sum_{i=0}^{\floor{h/2}}{h \choose 2i}
                      \sum_{j=0}^k{k \choose j}    \vol(G_{h,k},H_{i,j})    
                   \nonumber \\
               & =  \sum_{i=0}^{\floor{h/2}}{h \choose 2i}
                   \left(
                   \vol(G_{h,k},H_{i,0}) +\sum_{j=1}^k{k \choose j} \vol(G_{h,k},H_{i,1})
                   \right) \label{eq:lemma8}
                   \\
               & =  \sum_{i=0}^{\floor{h/2}} {h \choose 2i}
                   \left( 
                   \vol(G_{h,k},H_{i,0}) + (2^k-1) \vol(G_{h,k},H_{i,1})
                   \right) \nonumber \\
               & = \sum_{i=0}^{\floor{h/2}}
                   {h \choose 2i}
                   \left(
                      2^k\,\vol(G_{h,k},H_{i,0})
                   - (2^k-1)(\vol(G_{h,k},H_{i,0})-\vol(G_{h,k},H_{i,1}))
                   \right) \nonumber \\
               & = 2^k\, \sum_{i=0}^{\floor{h/2}} {h \choose 2i} \vol(G_{h,k},H_{i,0}) \nonumber \\
               & \qquad                    
                   - (2^k-1)
                   \sum_{i=0}^{\floor{h/2}} {h \choose 2i} (\vol(G_{h,k},H_{i,0})-\vol(G_{h,k},H_{i,1})), \nonumber\\ 
               & = 2^k\, \sum_{i=0}^{\floor{h/2}} {h \choose 2i} \vol(G_{h,k},H_{i,0}) \label{eq:parcial1} \\
               & \qquad                    
                 - (2^k-1) (\vol(G_{h,k},\Hzz)-\vol(G_{h,k},H_{0,1})), 
               \label{eq:parcial2}
 \end{align}
where~\eqref{eq:waka} is due to Lemma~\ref{lem:samepolynomials}, \eqref{eq:lemma8}
holds by Lemma~\ref{lem:onelooptype}, and \eqref{eq:parcial2} follows from Lemma~\ref{lem:somepathzeroloops}.
We remind the reader that $\vol(G_{h,k}, H_{0,0}) = \vol(G_{h,k}, \Hzz)$.

\section{The period for $\{1,3\}$-graphs}\label{sec:theperiod}

In this section we prove the main theorem of this paper.  We claim that we may henceforth consider only $t$ an even integer, due to the following argument.
Liu and Osserman~\cite[Lemma~3.3]{LiuO2006} proved that if $t$ is a nonnegative odd
integer, then $\vol(G_{h,k},H_{i,j}) = L^{\calP}_{h,k}(t)$ for every $i \geq 0$ and $j \geq 0$.
Applying their result for $t$ odd to the final relation of the previous section we see
that~\eqref{eq:parcial2} vanishes and from the
summation in~\eqref{eq:parcial1} we have that
\begin{align}
  L^{\calQ}_G(t) = L^{\calQ}_{h,k}(t)  & = 2^k\, \sum_{i=0}^{\floor{h/2}} {h \choose 2i} \vol(G_{h,k},H_{i,0}) \nonumber \\
  & = 2^k\, \sum_{i=0}^{\floor{h/2}} {h \choose 2i} L^{\calP}_{h,k}(t)  \nonumber \\
               & = N_{G_{h,k}} L^{\calP}_{h,k}(t) = N_G L^{\calP}_G(t) \label{eq:remark},
\end{align}
where~\eqref{eq:remark} follows from the fact that the number  of internal
Eulerian subgraphs of~$G_{h,k}$ is~$N_{G_{h,k}}= 2^k$ if $h = 0$ and $N_{G_{h,k}}=2^{k+h-1}$ if
$h > 0$.
From the relation $L^{\calQ}_G(t) = N_G L^{\calP}_G(t)$ for odd $t$,
and the fact that $L^{\calQ}_G(t)$ has period~1 or~2, 
Liu and Osserman concluded that the odd constituent polynomials $p_1$ and $p_3$
of $L^{\calP}_G$ are equal.

\subsection{The lemmas of the leaf-paths and of the loops}

For the purpose of determining the period of $L^{\calP}_{G}$, we settle the values of
the summation in~\eqref{eq:parcial1} and the difference in~\eqref{eq:parcial2}
for nonnegative even integers $t$ as stated in the two lemmas beneath.
Lemma~\ref{lem:summationleafs} is concerned with the summation in~\eqref{eq:parcial1} and 
Lemma~\ref{lem:summationloops} deals with the difference in~\eqref{eq:parcial2}.
We recall that for any $(h,k)$-caterpillar $G_{h,k}$  we have  $h+k \geq 2$, by definition.

\begin{lem}[the leaf-paths]\label{lem:summationleafs}
  Let $h$ and $k$ be  integers such that $h \geq 0$ and $k \geq 0$. 
  For every nonnegative even integer $t$,  we have that
  \begin{align}
    2^k\, \sum_{i=0}^{\floor{h/2}} {h \choose 2i} \vol(G_{h,k},H_{i,0})
    & \ = \ N_{G_{h,k}}\,\Big(L^{\calP}_{h,k}(t) + \sum_{j=1}^{\floor{h/2}} (-4)^{-j}\,{h-j \choose j}\frac{h}{h-j}\,L^{\calP}_{h{-}j,k}(t)\Big).  \nonumber
  \end{align}
\end{lem}

\begin{lem}[the loops]\label{lem:summationloops}
  Let $h$ and $k$ be integers such that $h \geq 0$ and $k \geq 1$. 
  For every nonnegative  even integer $t$,  we have that
  \begin{align}
    \vol(G_{h,k},\Hzz) - \vol(G_{h,k},H_{0,1}) & \ = \ \left\{\begin{array}{ll}
                                                 (\frac{t}{2}+1)^{k-1} & \mbox{if $t = 0 \Mod 4$ or $h = 0$}, \\
                                                 0            & \mbox{if $t = 2 \Mod 4$}.
                                               \end{array}\right. \nonumber
  \end{align}
\end{lem}

The proofs of theses lemmas are technical.
Therefore we shall defer the proof
of Lemma~\ref{lem:summationleafs} to Section~\ref{sec:leafs}
and of Lemma~\ref{lem:summationloops} to Section~\ref{sec:loops}.
In the remainder of this section we refer to these lemmas in order to deduce
the period of the Ehrhart quasi-polynomial in $t$ associated
to the polytope $\calP_G$. 

\subsection{Proof of the theorem of the period for $\{1,3\}$-graphs}

We have now all tools to deliver a theorem on the period collapse of the quasi-polynomial~$L^{\calP}_G(t)$.


\begin{proof}[Proof of Theorem~\ref{thm:main} (the period for $\{1,3\}$-graphs)]
  We denote by $q_0^{h,k}(t)$ and $q_1^{h,k}(t)$ the constituents of
  $L^{\calQ}_{h,k}(t)$ for $t$ even and $t$ odd, respectively.
  We also denote by $p^{h,k}_{j}(t)$ the constituent polynomials of $L^{\calP}_{h,k}(t)$ for $t = j \Mod 4$.
  
  Say $G$ has $h$ leaf-edges and cyclomatic number $k$.
  Note that $h+k \geq 2$.
  By Wakabayashi~\cite[Theorem~A(ii)]{Wakabayashi2013} and Lemma~\ref{lem:samepolynomials},
  we may assume that $G = G_{h,k}$.
  Liu and Osserman~\cite{LiuO2006} proved that $p^{h,k}_1 = p^{h,k}_3$.
  Therefore the period of $L^{\calP}_{h,k}$ is at most 2 if and only if $p^{h,k}_0 = p^{h,k}_2$.
    
  \medskip
  \noindent
  {\bf Case 1.} $k=0$.
 
  Then $G_{h,0}$ is a tree, and the period of $L^{\calP}_{h,0}(t)$ is~2 by
  Corollary~\ref{thm:period2trees}. 
  
  \medskip\noindent
  {\bf Case 2.} $h=0$.

  Then $G_{0,k}$ is a cubic graph, and the summation in~\eqref{eq:parcial1}
  reduces to the term associated with $i=0$. 
  By Lemma~\ref{lem:summationloops}, for every nonnegative even integer~$t$, we have that 
  \begin{align}
    L^{\calQ}_{0,k}(t) = q^{0,k}_0(t)& =  N_{G_{0,k}} L^{\calP}_{0,k}(t) - (2^k-1)\big(\frac{t}{2}+1\big)^{k-1} \nonumber \\
                                     & =  N_{G_{0,k}} L^{\calP}_{0,k}(t) - (N_{G_{0,k}}-1)\big(\frac{t}{2}+1\big)^{k-1}. \nonumber
  \end{align}
  %
  Therefore,
  $q_0^{0,k}(t) = N_{G_{0,k}}\,p_0^{0,k}(t) - (N_{G_{0,k}}-1) (\frac{t}{2}+1)^{k-1}$
  for $t = 0 \Mod{4}$.  Because both sides of the latter equation are
  polynomials in~$t$, the equality holds for every real~$t$.  Similarly,
  $q_0^{0,k}(t) = N_{G_{0,k}}\,p_2^{0,k}(t) - (N_{G_{0,k}}-1) (\frac{t}{2}+1)^{k-1}$
  for every real $t$.  Hence $p_0^{0,k} = p_2^{0,k}$ and the period of
  $L^{\calP}_{0,k}$ is at most~2.  

  \medskip\noindent
  {\bf Case 3.} $h \geq 1$ and $k \geq 1$.

  Now we tackle all other connected $\{1,3\}$-graphs.
  For $t = 0 \Mod 4$, applying  Lemma~\ref{lem:summationleafs}
  and Lemma~\ref{lem:summationloops} to the final relation of the previous section,
  we obtain that
  \begin{align}
     L^{\calQ}_{h,k}(t) = q^{h,k}_0(t) = \ & 
     N_{G_{h,k}}\,\Big(L^{\calP}_{h,k}(t) + \sum_{j=1}^{\floor{h/2}} (-4)^{-j}\,{h-j \choose j}\frac{h}{h-j}\,L^{\calP}_{h{-}j,k}(t)\Big)\nonumber \\
     &    - (2^k-1)\big(\frac{t}2+1\big)^{k-1}. \label{eq:mod0}
  \end{align} 
  Similarly, for $t = 2 \Mod 4$, we can derive that
  \begin{align}
        L^{\calQ}_{h,k}(t) = q^{h,k}_0(t) & =  N_{G_{h,k}}\,\Big(L^{\calP}_{h,k}(t) + \sum_{j=1}^{\floor{h/2}} (-4)^{-j}\,{h-j \choose j}\frac{h}{h-j}\,L^{\calP}_{h{-}j,k}(t)\Big).    
    \label{eq:mod2}
  \end{align} 

  
 Using \eqref{eq:mod0} with $t = 0 \Mod 4$,  and substituting $N_{G_{h,k}} =  2^{h+ k -1}$, we have
  \begin{align}
    \frac{q^{h,k}_0(t)}{N_{G_{h,k}}} & =  p_0^{h,k}(t) + \sum_{j=1}^{\floor{h/2}} (-4)^{-j}\,{h-j \choose j}\frac{h}{h-j}\,p_0^{h{-}j,k}(t) - \frac{1}{2^{h-1}} (1 -\frac1{2^k}   )\,\big(\frac{t}2+1\big)^{k-1}. \label{eq:mod0poly} 
  \end{align} 
  Because both sides of Equality~\eqref{eq:mod0poly} are polynomials in $t$,
  it holds for every~$t$.  Similarly, from Equality~\eqref{eq:mod2}, we deduce that, for every~$t$, 
  \begin{eqnarray}\label{eq:mod2poly}
    \frac{q^{h,k}_0(t)}{N_{G_{h,k}}} & = & p_2^{h,k}(t) + \sum_{j=1}^{\floor{h/2}} (-4)^{-j}\,{h-j \choose j}\frac{h}{h-j}\,p_2^{h{-}j,k}(t).
  \end{eqnarray} 
             
  By subtracting~\eqref{eq:mod0poly} from~\eqref{eq:mod2poly} and rearranging the terms, we deduce that
  \begin{align}\label{eq:dif}
    (p^{h,k}_0-p^{h,k}_2)(t) + \sum_{j=1}^{\floor{h/2}} (-4)^{-j}\,{h-j \choose j} & \frac{h}{h-j}\,(p^{h-j,k}_0-p^{h-j,k}_2)(t) \nonumber \\
    & \ = \ \frac{1}{2^{h-1}} \left(1 - \frac1{2^k}\right)\,\left(\frac{t}2+1\right)^{k-1}.
  \end{align} 
We define
  \begin{equation}\label{eq:defi_d}
    d(\alpha) \ := \ \frac{(-4)^{\alpha}}{\alpha}\,\frac{(p^{\alpha,k}_0 - p^{\alpha,k}_2)(t)}{(1-\frac1{2^k})\,(\frac{t}2+1)^{k-1}},
  \end{equation}
so that we may divided \eqref{eq:dif} by~$\frac{h}{(-4)^h}(1-\frac{1}{2^k})(\frac{t}{2}+1)^{k-1}$ 
and obtain
  \begin{equation}\label{eq:recorrence} 
    \sum_{j=0}^{h-1}{h-j \choose j}\,d(h-j) \ = \ \frac{(-4)^h}{h \, 2^{h-1}} \ = \ \frac{(-1)^h \, 2^{h+1}}{h}. 
  \end{equation}
  The $j=\floor{h/2}+1,\ldots,h-1$ terms are all equal to zero.
  
  From one of the so called simpler Chebyshev inverse relations~\cite[Table~2.3, item~5]{Riordan1968}, we derive that
  \begin{eqnarray*} 
    d(h) & = & \sum_{j=0}^{h-1}\left[{h+j-1 \choose j} - {h+j-1 \choose j-1}\right] \frac{(-1)^j{(-1)^{h-j}2^{h-j+1}}}{h-j} \\
         & = & \sum_{j=0}^{h-1}\frac{h-j}{h+j}{h+j \choose j} \frac{(-1)^h 2^{h-j+1}}{h-j} \\
         & = & (-1)^h\,\sum_{j=0}^{h-1}\frac{2^{h-j+1}}{h+j}\,{h+j \choose j}.
  \end{eqnarray*}
  Thus, we have that $d(h) > 0$ for every positive even $h$ and $d(h) < 0$ for every odd $h$. 
  Finally, substituting in~\eqref{eq:defi_d} with $\alpha=h$, we obtain that
  \begin{equation*}
    p^{h,k}_0(t) - p^{h,k}_2(t) \ = \ \frac{h}{4^h} \left(1-\frac1{2^k}\right)\,\left(\frac{t}{2}+1\right)^{k-1}\,\sum_{j=0}^{h-1}\frac{2^{h-j+1}}{h+j}\,{h+j \choose j} > 0
  \end{equation*}
  and conclude that $p^{h,k}_0(t) > p^{h,k}_2(t)$ for every $t \geq 0$.  
  Consequently, the period of~$L^{\calP}_{h,k}$ is~4.
\end{proof}


\section{Lemma of the leaf-paths}\label{sec:leafs}

The purpose of this section is to prove Lemma~\ref{lem:summationleafs}, the lemma of the leaf-paths.
To achieve this goal, we first establish the necessary tools in two supporting lemmas and a combinatorial identity.
We defer the proofs of these lemmas to the end of this section.
 
Then, in Section \ref{bijections section}, we give a tool that is purely set-theoretic, namely  
Lemma \ref{lem:OI}.
This tool allows us to shift our lattice, in order to measure the difference
between two discrete volumes.    This is also the main tool that is employed in the
proofs of the supporting lemmas of this section, and in the following section.
Finally, we end this section by giving the proofs of the supporting lemmas.

Before continuing we would like to recall that $H_{i,j}$  denotes any internally Eulerian subgraph
with $i$ leaf-paths and $j$ loops and that for any $(h,k)$-caterpillar $G_{h,k}$  we have
$h+k \geq 2$, by definition.

\subsection{Supporting lemmas, and an identity}

Here, we state, without proving, the lemmas and identity used in the proof of the lemma of the leaf-paths.

\begin{lem}[one less leaf-path]\label{lem:onelessleafpath}
  Let $h$ and $k$ be integers such that $h \geq 2$ and $k\geq 1$.
  For every nonnegative even integer $t$ and for every integer $i = 1,\ldots,\floor{h/2}$,
  we have that
  \begin{align}
     \label{eq:onelessleafpath}
     \vol(G_{h,k},H_{i,0}) & \ = \ \vol(G_{h,k},H_{i-1,0}) - \vol(G_{h-1,k},H_{i-1,0}). 
  \end{align}
\end{lem}

\noindent
From this lemma, we can deduce the following identity.

\begin{lem}[the size of coset $H_{i,0}$]\label{lem:leafpaths}
  Let $h$ and $k$ be integers such that $h \geq 0$ and $k \geq 1$. 
  For every nonnegative even integer $t$ and for every integer $i=0,\ldots,\floor{h/2}$,
  we have that
\begin{align}
 \vol(G_{h,k},H_{i,0}) & \ = \ \sum_{j=0}^i (-1)^j\,{i \choose j}\, L^{\calP}_{h-j,k}(t).
                             \nonumber
\end{align}
\end{lem}

We shall derive the lemma of the leaf-paths from Lemma~\ref{lem:leafpaths}
and the following known combinatorial identity.
  \begin{identity}{{\cite[Problem~18, Chapter~6]{Riordan1968}}}\label{clm:binomials}
    For every integer $h \geq 1$ and $j=0,\ldots,\floor{h/2}$, 
    \begin{align}
      \sum_{i=j}^{\floor{h/2}}{h \choose 2i}{i \choose j} \ = \ {h-j \choose j}\frac{h}{h-j}\,2^{h-1-2j}.
      \nonumber
    \end{align}
  \end{identity}

\subsection{Proof of the lemma of the leaf-paths}

We are now prepared to carry out the proof of Lemma~\ref{lem:summationleafs},
upon which the proof of Theorem~\ref{thm:main} of the period of $L^{\calP}_G(t)$ rests.

\vspace{2mm}
\begin{proof}[Proof of Lemma~\ref{lem:summationleafs} (the leaf-paths)]
 The lemma clearly holds for $h=0$. Thus we may assume $h \geq 1$ and we have that
  \begin{align}
    2^k \,  \sum_{i=0}^{\floor{h/2}} {h \choose 2i} & \vol(G_{h,k},H_{i,0}) 
                        =  2^k\,\sum_{i=0}^{\floor{h/2}}
                         {h \choose 2i} \sum_{j=0}^i (-1)^j\,{i \choose j}\, L^{\calP}_{h{-}j,k}(t) \label{eq:newlemma} \\ 
            & =  2^k\,\sum_{i=0}^{\floor{h/2}} {h \choose 2i}\left(L^{\calP}_{h,k}(t)
              +\sum_{j=1}^i (-1)^j\,{i \choose j}\,L^{\calP}_{h{-}j,k}(t) \right) \nonumber \\ 
                      & =  2^{k+h-1}\,L^{\calP}_{h,k}(t)
                        + 2^k\,\sum_{i=1}^{\floor{h/2}} {h \choose 2i}
              \sum_{j=1}^i (-1)^j\,{i \choose j}\,L^{\calP}_{h{-}j,k}(t) \nonumber \\
                      & =  N_{G_{h,k}}\,L^{\calP}_{h,k}(t)
                        + 2^k\,\sum_{j=1}^{\floor{h/2}} (-1)^j\,L^{\calP}_{h{-}j,k}(t)
                        \sum_{i=j}^{\floor{h/2}}{h \choose 2i}{i \choose j} \label{eq:exchangesums} \\
                      & = N_{G_{h,k}}\,L^{\calP}_{h,k}(t)
                        + 2^k\,\sum_{j=1}^{\floor{h/2}}
                        (-1)^j\,L^{\calP}_{h{-}j,k}(t)
                        {h-j \choose j}\frac{h}{h-j}   
                        \,2^{h-1-2j} \label{eq:riordan} \\
            & = N_{G_{h,k}}\,L^{\calP}_{h,k}(t)
              + 2^{k+h-1}\,\sum_{j=1}^{\floor{h/2}}
              (-1)^j \,2^{-2j}
              {h-j \choose j}\frac{h}{h-j}   
              \,L^{\calP}_{h{-}j,k}(t) \nonumber\\
            & = N_{G_{h,k}}\,L^{\calP}_{h,k}(t)
              + N_{G_{h,k}}\,\sum_{j=1}^{\floor{h/2}}
              (-4)^{-j}
              {h-j \choose j}\frac{h}{h-j}   
              \,L^{\calP}_{h{-}j,k}(t) \nonumber\\
            & = N_{G_{h,k}}\, \Big(L^{\calP}_{h,k}(t)
              + \sum_{j=1}^{\floor{h/2}}
              (-4)^{-j}
              {h-j \choose j}\frac{h}{h-j}   
              \,L^{\calP}_{h{-}j,k}(t)\Big), \nonumber
\end{align} 
where~\eqref{eq:newlemma} holds by Lemma~\ref{lem:leafpaths},
\eqref{eq:exchangesums} is the result of exchanging the two summations, and
finally~\eqref{eq:riordan} is due to~Identity~\ref{clm:binomials}.
\end{proof}

\subsection{Bijections by shifting}\label{bijections section}

Here we derive a purely set-theoretic tool, Lemma \ref{lem:OI},
that gives a partial bijection \textit{by shifting}.    
We will employ this tool repeatedly in the
remaining proofs of this section, and in those of the next section.

We begin with some notation that is followed by the lemma. 
For an $e$ in $E$ and a set $\calX \subset \IR^E$, let
\begin{eqnarray*}
  \Out_e(\calX) & := & \{w \in \calX\cap \ZZ^E: w_e \mbox{ is even and } w+\um_e \not\in \calX\} \mbox{\ and} \\
  \In_e(\calX)  & := & \{w \in \calX\cap \ZZ^E: w_e \mbox{ is odd and } w-\um_e \not\in \calX\}.
\end{eqnarray*}
We think of the function~$w \mapsto w+\um_e$ as a shifting of the integer points with
even $e$-coordinate to the integer points with odd $e$-coordinate.
$\Out_e(\calX)$ is the set of integer points of $\calX$ with even $e$-coordinate
that are  moved outside $\calX$ and~$\In_e(\calX)$ is the set of integer points with even
$e$-coordinate that are moved inside $\calX$ (Figure~\ref{fig:latticeshift}).

\begin{figure}[htbp]
  \centering\includegraphics[width=0.95\linewidth]{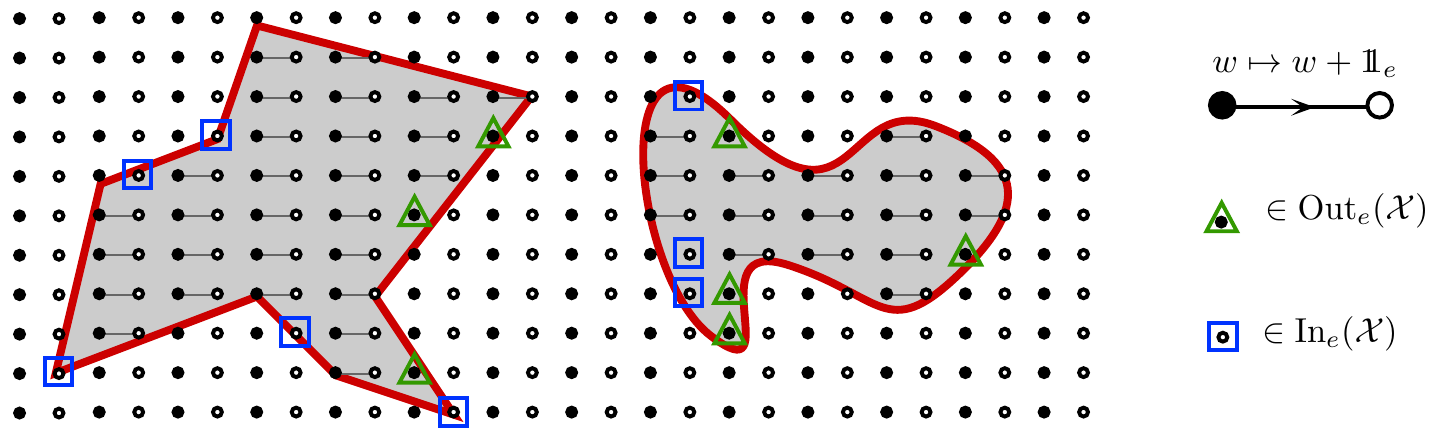}
\caption{Filled circles represent points with an even $e$-coordinate, 
  while empty circles represent points with an odd $e$-coordinate.  
  Let $\calX$ be the set of points in the grey area.  The circles marked by triangles 
  represent the set $\Out_e(\calX)$ and the circles marked by squares represent $\In_e(\calX)$.  
  The difference between the number of filled and empty circles in~$\calX$ is 
  ${|\Out_e(\calX)| - |\In_e(\calX)| = 7 - 8 = -1}$.}
\label{fig:latticeshift}
\end{figure}

\newpage

\noindent
For any subset $\calX \subset \IR^{E}$, we define
\[
 \volev{e}(\calX) := |  \{  w \in \calX \cap \ZZ^E :  w_e  \text{ is even}  \}  |, 
 \]
and 
\[
\volod{e}(\calX) := |  \{  w \in \calX \cap \ZZ^E :  w_e  \text{ is odd}  \}  |.
 \]

\begin{lem}[the bijection by shifting]\label{lem:OI}
  If $\calX \subset \IR^E$ is bounded and $e \in E$, then 
  $$\volev{e}(\calX) - \volod{e}(\calX) \ = \ |\Out_e(\calX)| - |\In_e(\calX)|.$$
\end{lem}

\begin{proof}
  We have that
  \[\Out_e(\calX) \subseteq \{w \in \calX \cap \ZZ^E : w_e \mbox{ is even}\} \mbox{ and }
  \In_e(\calX) \subseteq  \{w \in \calX \cap \ZZ^E : w_e \mbox{ is odd}\}.
  \]
  Therefore,
  \begin{eqnarray}
    \volev{e}(\calX)    & = & |\{w \in \calX \cap \ZZ^E : w_e \mbox{ is even}\} \setminus \Out_e(\calX)| + |\Out_e(\calX)|, \mbox{ and} \label{um} \\    
    \, \volod{e}(\calX) & = & |\{w \in \calX \cap \ZZ^E : w_e \mbox{ is odd}\} \setminus \In_e(\calX)| + |\In_e(\calX)|. \label{dois}
  \end{eqnarray}
  The function ${w \mapsto w+\um_e}$ is a bijection between 
  $\{w \in \calX \cap \ZZ^E: w_e \mbox{ is even}\} \setminus \Out_e(\calX)$ 
  and $\{w \in \calX \cap \ZZ^E : w_e \mbox{ is odd}\} \setminus \In_e(\calX)$. 
  As $\calX$ is bounded, $\volev{e}(\calX)$ and $\volod{e}(\calX)$
  are finite and the result follows from \eqref{um} and \eqref{dois}.
\end{proof}

\subsection{Proofs of the supporting lemmas}

We begin this section by presenting the proof of the supporting Lemma~\ref{lem:onelessleafpath} of one less leaf-path.

\vspace{2mm}
\begin{proof}[Proof of Lemma~\ref{lem:onelessleafpath} (one less leaf-path)]
  By Lemma~\ref{lem:sametypecosets} we may assume that one of the leaf-paths $P$ in $H_{i, 0}$
  consists of two stems $a$ and $b$ of $G_{h,k}$ incident to a vertex~$r$.
  Let $u$ and $v$ be the leaves adjacent to~$a$ and~$b$, respectively (Figure~\ref{fig:inductionP}(a)).
  Let $G_{h-1,k} := G_{h,k}-\{u, v\}$ and let $c$ be the edge incident to $r$
  other than~$a$ or~$b$ (Figure~\ref{fig:inductionP}(b)). 
  Hence,  $H_{i-1, 0} := H_{i,0} - P$ is an internally Eulerian subgraph of
  $G_{h,k}$ and $G_{h-1,k}$.

\begin{figure}[hbt]
   \centering\scalebox{0.9}{\includegraphics{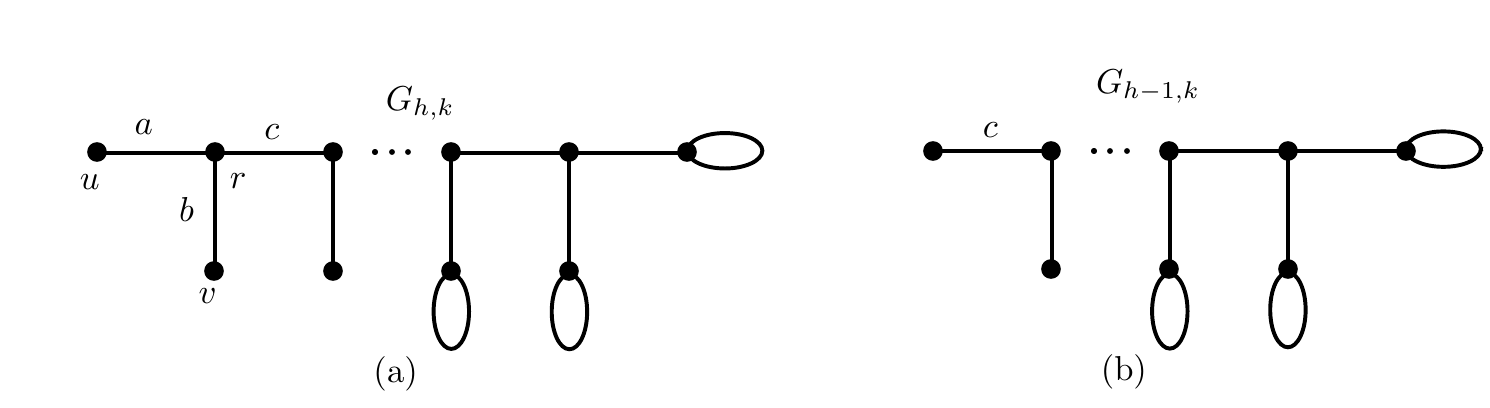}}
  \caption{(a) $G_{h,k}$ and the stems $a$ and $b$. 
    (b) $G_{h-1,k}$.}
  \label{fig:inductionP}
\end{figure}

For conciseness, let $W:=\parte(G_{h,k},H_{i, 0})$, $Z:=\parte(G_{h,k},H_{i-1, 0})$, and
$X :=  \parte(G_{h-1,k}, H_{i-1,0})$. Equality~\eqref{eq:onelessleafpath}
  is equivalent to $|X| = |Z| - |W|$.
  We partition $Z$ and $W$ into
  \begin{align}
    Z_< & := \{(w,z) \in Z: z_r < t\}, \nonumber \\
    Z_= & := \{(w,z) \in Z: z_r = t\},  \nonumber \\
    W_< & := \{(w,z) \in W: w_c < w_a+w_b\} \mbox{ and} \nonumber \\
    W_= & := \{(w,z) \in W: w_c = w_a+w_b\}. \nonumber
  \end{align}

Let $I$ be the set of internal nodes of~$G_{h,k}$, and consider 
the function ${\phi_{+} : \IR^E \times \IR^I \rightarrow \IR^E \times \IR^I}$, given  by the shifting operation
\begin{align}
   \phi_{+}(w,z) & = (w,z) + (\um_{a,b},\um_r). \nonumber
\end{align}
One can verify that $\phi_{+}$ is a bijection between $Z_<$ and $W_<$ and therefore $|Z_<| = |W_<|$.
Consequently, in order to prove the lemma, it suffices to show that
\[|X| = |Z|-|W| = (|Z_<|+|Z_=|) - (|W_<| + |W_=|) = |Z_=|-|W_=|.\]
Conceptually, from the viewpoint of shiftings and Lemma~\ref{lem:OI},
$Z_=$ is reminiscent of $\Out_{ab}(Z)$ and $W_=$ 
is reminiscent of $\In_{ab}(W)$.

For $j=0,1,\ldots,t/2$, let
    \begin{align}
    X_j  & := \{(w,z) \in X : w_c = 2j\} \nonumber
  \end{align}
  and, for $j=0,1,\ldots,t/2$ and $\ell = 0,1,\ldots,j$, let
\begin{align}  
  Z_{j,\ell}  & := \{(w,z) \in Z_= : w_a = t-2\ell,\ w_b = t - 2j + 2\ell,\ w_c = 2j\}. \nonumber
\end{align}
The edge $c$ is not in $H_{i-1,0}$, thus the sets $X_0,\ldots,X_{t/2}$ form a partition of $X$.
The sets $Z_{0,0}, Z_{1,0},\ldots,Z_{t/2,t/2}$
form a partition of $Z_=$.
Moreover, for each $j$, we have that ${|Z_{j, \ell}| = |X_j|}$, for every $\ell$. Therefore,
\[
  |Z_=| \ = \ \sum_{j=0}^{t/2} \sum_{\ell=0}^j |Z_{j,\ell}| \ = \ \sum_{j=0}^{t/2} \sum_{\ell=0}^j |X_j| \ = \ \sum_{j=0}^{t/2} (j+1)|X_j|.
\]  

For $j=0,1,\ldots,t/2$ and $\ell = 1,\ldots,j$, we define
\begin{align}  
    W_{j,\ell}  & := \{(w,z) \in W_= : w_a = 2\ell-1,\ w_b = 2j-2\ell+1,\  w_c = 2j\}. \nonumber
\end{align}
The sets $W_{0,1},\ldots,W_{t/2,t/2}$ form a partition of $W_=$ and 
$|W_{j, \ell}| = |X_j|$ for each~$j$ and each~$\ell$. Hence,
\[
  |W_=| \ = \ \sum_{j=0}^{t/2} \sum_{\ell=1}^j |W_{j,\ell}| \ = \ \sum_{j=0}^{t/2} \sum_{\ell=1}^j |X_j| \ = \ \sum_{j=0}^{t/2} j|X_j|.
\]  
Therefore, $|Z_=|-|W_=| = \sum_{j=0}^{t/2} |X_j| = |X|$.
\end{proof}

Finishing this section we handle the proof of the supporting Lemma~\ref{lem:leafpaths}
of the size of coset $H_{i,0}$.

\vspace{2mm}
\begin{proof}[Proof of Lemma~\ref{lem:leafpaths} (the size of coset $H_{i,0}$)]
  The proof is by induction on $i$, the number of leaf-paths in $H_{i,0}$.
  For $i = 0$, the lemma follows from~\eqref{eq:emptycoset}.
  Hence, we may assume that~$i \geq 1$ and thus $h \geq 2$.
  We then have that
  \begin{align}
  \vol(&G_{h,k},H_{i,0})  =  \vol(G_{h,k},H_{i-1,0}) - \vol(G_{h-1,k},H_{i-1,0}) \label{eq:001} \\
       & =  \sum_{j=0}^{i-1} (-1)^j\,{i-1 \choose j}\, L^{\calP}_{h-j,k}(t)  
         - \sum_{j=0}^{i-1} (-1)^j\,{i-1 \choose j}\, L^{\calP}_{h-1-j,k}(t) 
         \quad \quad \label{eq:002} \\
       & =  \sum_{j=0}^{i-1} (-1)^j\,{i-1 \choose j}\,  L^{\calP}_{h-j,k}(t) 
         - \sum_{j=1}^i (-1)^{j-1}\,{i-1 \choose j-1}\,  L^{\calP}_{h-j,k}(t) 
      \quad \label{eq:003} \\
       & =  \sum_{j=0}^{i-1} (-1)^j\,{i-1 \choose j}\, L^{\calP}_{h-j,k}(t) 
         + \sum_{j=1}^i (-1)^j\,{i-1 \choose j-1}\, L^{\calP}_{h-j,k}(t) \nonumber 
    \\
       & =  \sum_{j=0}^i (-1)^j\,{i \choose j}\, L^{\calP}_{h-j,k}(t), \nonumber 
  \end{align}
  where~\eqref{eq:001} holds by Lemma~\ref{lem:onelessleafpath}, \eqref{eq:002} 
  follows by induction, and~\eqref{eq:003} holds by changing $j$ to~$j+1$ in the second sum.
\end{proof}



\section{Lemma of the loops}\label{sec:loops}

This section is dedicated to establish Lemma~\ref{lem:summationloops} of the loops. 
As always, $H_{i,j}$  denotes any internally Eulerian subgraph of $G_{h,k}$ 
with $i$ leaf-paths and $j$ loops.  To handle the boundary conditions on the indices, we recall that  $H_{0,0} = \Hzz$.  Moreover, for any $(h,k)$-caterpillar~$G_{h,k}$  we have  $h+k \geq 2$, by definition.

\subsection{Supporting lemmas}

The proof of Lemma~\ref{lem:summationloops} of the loops relies on two supporting lemmas.
For conciseness, we recall that if~$\calX$ is a subset of $\IR^{E}$ we write
$\volev{a}(\calX)$ to denote the number of integer points~$w$ in~$\calX$ such that~$w_a$ is even and
we write $\volod{a}(\calX)$ to denote the number of integer points~$w$ in~$\calX$ such that~$w_a$ is odd.

\begin{lem}[$\calQ$ to $\calP$]\label{lem:reduction}
   Let $h$ and $k$ be integers such that  $h \geq 0$ and $k \geq 1$ and
   let~$a$ be a leaf-edge in $G_{h+1,k-1}$.       
   For every nonnegative even integer~$t$, we have that
   \[
     \vol(G_{h,k},\Hzz) - \vol(G_{h,k},H_{0,1})
     \ = \  \volev{a}(t\calP_{G_{h+1,k-1}}) - \volod{a}(t\calP_{G_{h+1,k-1}}) .
   \]
\end{lem}

\begin{lem}[the evaluation]\label{lem:rootedloopedtrees}
   Let $h$ and $k$ be integers such that  $h \geq 1$ and $k \geq 0$ and
   let~$a$ be a leaf-edge in $G_{h,k}$.
   For every nonnegative even integer $t$,  we have that
   \[
     \volev{a}(t\calP_{G_{h,k}}) - \volod{a}(t\calP_{G_{h,k}})
     \ = \ \left\{\begin{array}{ll}
                           (\frac{t}{2}+1)^{k} & \mbox{if $t = 0 \Mod 4$  or $h = 1$}, \\
                           0            & \mbox{if $t = 2 \Mod 4$}.
                                                    \end{array}\right.
   \]
\end{lem}

\subsection{Proof of the lemma of the loops}

Lemma~\ref{lem:summationloops} is a straight consequence of the supporting lemmas.

\vspace{2mm}
\begin{proof}[Proof of Lemma~\ref{lem:summationloops} (the loops)]
  Because $h$ and $k$ are integers such that $h \geq 0$ and $k \geq 1$, $G_{h+1,k-1}$ has a leaf-edge $a$.   We have that
  \begin{align}
     \vol(G_{h,k},\Hzz) - \vol(G_{h,k},H_{0,1})
     & = \  \volev{a}(t\calP_{G_{h+1,k-1}}) - \volod{a}(t\calP_{G_{h+1,k-1}}) \nonumber \\
     & = \ \left\{\begin{array}{ll}
                           (\frac{t}{2}+1)^{k+1} & \mbox{if $t = 0 \Mod 4$  or $h = 0$}, \\
                           0            & \mbox{if $t = 2 \Mod 4$},
                                                    \end{array}\right. \nonumber
  \end{align}
  where the first equality is by Lemma~\ref{lem:reduction} and, 
  noting that $h+1 \geq 1$ and $k-1 \geq 0$, 
  the second equality is obtained by applying Lemma~\ref{lem:rootedloopedtrees} to $G_{h+1,k-1}$. 
\end{proof}

\subsection{Proofs of the supporting lemmas}

\newcommand{\Hzo}{\ell}

We begin by proving Lemma~\ref{lem:reduction} that moves our attention from the
polytope $\calQ_{G_{h,k}}$ to the polytope $\calP_{G_{h,k}}$. 

\begin{proof}[Proof of Lemma~\ref{lem:reduction} ($\calQ$ to $\calP$)]
Let $a$ be an edge of $G_{h,k}$, and let $\ell$ be a loop of $G_{h,k}$, both incident with the same node $r$. 
By Lemma~\ref{lem:sametypecosets}, we may assume that $H_{0,1}$ is the subgraph induced by the loop $\ell$
and that the configuration is as depicted in Figure~\ref{fig:massage}(a). In the proof we write $\Hzo$ instead of $H_{0,1}$.
This lemma is concerned with the evaluation of 
\[
    \vol(G_{h,k},\Hzz) - \vol(G_{h,k},\Hzo) = |\parte(G_{h,k},\Hzz)| - |\parte(G_{h,k},\Hzo)|.
\]
\begin{figure}[hbt]
  \centering\scalebox{0.9}{\includegraphics{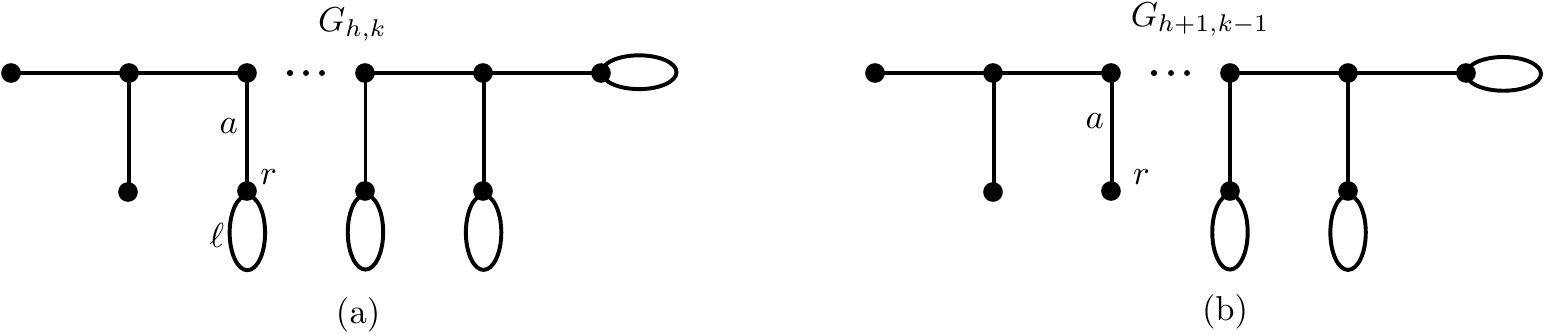}}
  \caption{(a) $G_{h,k}$ and the edge $a$, loop $\ell$ and node $r$. 
    (b) $G_{h+1,k-1}$ obtained from $G_{h,k}$ by deleting $\ell$.}
  \label{fig:massage}
\end{figure}

For $\alpha = 0,\ldots,t$, let
  \begin{align}
    \ZZeven{\ell}{\alpha} & := \{(w,z) \in \parte(G_{h,k},\Hzz): z_r = \alpha\}, \mbox{ and}\nonumber \\
    \ZZodd{\ell}{\alpha}  & := \{(w,z) \in \parte(G_{h,k},\Hzo): z_r = \alpha\}.  \nonumber 
  \end{align} 
The sets $\ZZeven{\ell}{0}, \ldots, \ZZeven{\ell}{t}$ form a partition of  $\parte(G_{h,k},\Hzz)$ and the sets
$\ZZodd{\ell}{0}, \ldots, \ZZodd{\ell}{t}$ form a partition of $\parte(G_{h,k},\Hzo)$.
Accordingly,
  \begin{align}
    \vol(G_{h,k},\Hzz)   & = |\ZZeven{\ell}{0}|  + \cdots + |\ZZeven{\ell}{t}|, \mbox{ and} \label{evenl} \\
    \vol(G_{h,k},\Hzo)   & = |\ZZodd{\ell}{0}|   + \cdots + |\ZZodd{\ell}{t}|.   \label{oddl}
  \end{align}
  
Let $E$ be the set of edges and $I$ be the set of internal nodes of~$G_{h,k}$. Consider 
the function ${\phi^{}_{+} : \IR^E \times \IR^I \rightarrow \IR^E \times \IR^I}$, given by the shifting operation
\begin{align}
   \phi^{}_{+}(w,z) & = (w,z) + (\um_{\ell},\um_r). \nonumber
\end{align}
One can verify that for every $\alpha$ even, $\alpha < t$, the
function $\phi^{}_{+}$ is a bijection between $\ZZeven{\ell}{\alpha}$
and $\ZZodd{\ell}{\alpha+1}$.  
Therefore, for every $\alpha$ even, $\alpha < t$, we have that
$|\ZZeven{\ell}{\alpha}| = |\ZZodd{\ell}{\alpha+1}|$.

Similarly, consider 
the function ${\phi^{}_{-} : \IR^E \times \IR^I \rightarrow \IR^E \times \IR^I}$, given by  the reverse shifting
\begin{align}
   \phi_{-}(w,z) & = (w,z) - (\um_{\ell},\um_r). \nonumber
\end{align}
One can also check that for every $\alpha$ odd, $\alpha \leq t$, the function $\phi^{}_{-}$ is a bijection between $\ZZeven{\ell}{\alpha}$ and $\ZZodd{\ell}{\alpha-1}$.
Consequently, for every $\alpha$ odd, $\alpha \leq t$, we have that $|\ZZeven{\ell}{\alpha}| = |\ZZodd{\ell}{\alpha-1}|$.

Keeping in mind the induced bijections given by the shiftings $\phi_+$ and $\phi_-$, and using Equalities
 \eqref{evenl} and \eqref{oddl},
 we can see that for every nonnegative even integer $t$:
\begin{align}
  \vol(G_{h,k},\Hzz)- \vol(G_{h,k},\Hzo)
  & = |\ZZeven{\ell}{0}|  + \cdots + |\ZZeven{\ell}{t}| - |\ZZodd{\ell}{0}|   - \cdots - |\ZZodd{\ell}{t}| \nonumber \\ 
  & =       (|\ZZeven{\ell}{0}| - |\ZZodd{\ell}{1}|) + (|\ZZeven{\ell}{1}| - |\ZZodd{\ell}{0}|) + \cdots   \nonumber \\ 
  & \quad + (|\ZZeven{\ell}{t-2}| - |\ZZodd{\ell}{t-1}|) + (|\ZZeven{\ell}{t-1}| - |\ZZodd{\ell}{t-2}|) \nonumber \\
  & \quad + (|\ZZeven{\ell}{t}| - |\ZZodd{\ell}{t}|) \nonumber \\
  & = |\ZZeven{\ell}{t}| - |\ZZodd{\ell}{t}|. \nonumber
\end{align}
Our next and last step in this proof is showing that for every nonnegative even integer $t$:
\begin{align}\label{eq:bijections}
  |\ZZeven{\ell}{t}| - |\ZZodd{\ell}{t}| & =  \volev{a}(t\calP_{G_{h+1,k-1}}) - \volod{a}(t\calP_{G_{h+1,k-1}}).
\end{align}
For this we outline bijections
between the sets
$\ZZeven{\ell}{t} \mbox{ and } \{w' \in t\calP_{G_{h+1,k-1}}: w'_a \mbox{ is even}\}$
and between the sets
$\ZZodd{\ell}{t}  \mbox{ and } \{w' \in t\calP_{G_{h+1,k-1}}: w'_a \mbox{ is odd}\}.$

The $(h+1,k-1)$-caterpillar $G_{h+1,k-1}$ will be seen as resulting from the deletion of the loop $\ell$ from $G_{h,k}$.
Because we denote the edge set of $G_{h,k}$ by $E$, then the edge set of $G_{h+1,k-1}$ is  $E \setminus \{\ell\}$.
If $(w,z)$ is a point in $\ZZeven{\ell}{t} \cup \ZZodd{\ell}{t}$, then
$w_e$ is an even integer for each $e$ in $E \setminus \{\ell\}$.
Moreover, as each point $(w,z) \in \calQ_{G_{h,k}}$ satisfies $w_a +
2w_{\ell} = 2z_r$, if $(w,z)$ is a point in $\ZZeven{\ell}{t}$, then
$w_a = 0 \Mod 4$ and if $(w,z)$ is a point in $\ZZodd{\ell}{t}$, then
$w_a = 2 \Mod 4$.

Consider the function from $\IR^E \times \IR^I$ to $\IR^{E \setminus \{\ell\}}$ given by $(w,z) \mapsto w'$, where
$w'_e = w_e/2$ for each $e$ in $E \setminus \{\ell\}$.
From the above observations concerning parity,
one can verify that $(w,z) \mapsto w'$ injectively maps $\ZZeven{\ell}{t}$ into
$\{w' \in t\calP_{G_{h+1,k-1}}: w'_a \mbox{ is even}\}$ and 
$\ZZodd{\ell}{t}$ into $\{w' \in t\calP_{G_{h+1,k-1}}: w'_a \mbox{ is odd}\}$.

Conversely, the function from $\IR^{E \setminus \{\ell\}}$ to $\IR^E \times \IR^I$ given by $w' \mapsto (w,z)$,
where firstly we set $w_e = 2w'_e$ for each $e$ in $E \setminus \{\ell\}$ and set $w_{\ell} = t - w'_a$,
and secondly for each internal node~$v$ of $G_{h,k}$ we set $z_v$ to the sum of the components $w$ associated to the edges incident to~$v$,
adding twice the value associated to loops. Once more, one can check that $w' \mapsto (w,z)$  provides the inverse injective function.
\end{proof}

\begin{remark} 
From the previous proof, one can derive that for every nonnegative odd integer $t$:
\begin{align}
   \vol(G_{h,k},\Hzz)- \vol(G_{h,k},\Hzo) & = |\ZZeven{\ell}{0}|  + \cdots + |\ZZeven{\ell}{t}| - |\ZZodd{\ell}{0}|   - \cdots - |\ZZodd{\ell}{t}| \nonumber\\ 
   & = (|\ZZeven{\ell}{0}| - |\ZZodd{\ell}{1}|) + (|\ZZeven{\ell}{1}| - |\ZZodd{\ell}{0}|) + \cdots \nonumber \\ 
   & \quad + (|\ZZeven{\ell}{t-1}| - |\ZZodd{\ell}{t}|)  +  (|\ZZeven{\ell}{t}| - |\ZZodd{\ell}{t-1}|)  \nonumber \\
   & = 0. \nonumber
\end{align}
This equality also follows from the more generic lemma of Liu and Osserman~\cite[Lemma~3.3]{LiuO2006}.\qed
\end{remark}

It remains to prove Lemma~\ref{lem:rootedloopedtrees} of the evaluation.
Consider an $(h,k)$-caterpillar $G_{h,k}$ with $h \geq 1$.
Let $r$ be a leaf of $G_{h,k}$, let~$a$ be the edge incident to $r$,
let $q$ be the other end node of $a$.
There are two other edges, not necessarily distinct,  $b$ and $c$ incident to~$q$.
If $w$ is a point in $\Out_a(t\calP_{G_{h,k}})$ then $w_a + w_b + w_c = t$ or $w_a= w_b + w_c$.
Therefore, $w_a = \min\{t-w_b-w_c, w_b+w_c\}$.
Similarly, if $w$ is a point in $\In_a(t\calP_{G_{h,k}})$, then $w_b = w_a + w_c$ or
$w_c = w_a + w_b$. Thus, $w_a = \max\{w_b-w_c, w_c-w_b\}$. 
This digression is summarized in the following lemma.

\begin{lem}\label{lemma:OeIe}
  If $a$ is a leaf-edge of $G_{h,k}$ and $b$ and $c$ are two other edges incident to $a$, 
  not necessarily distinct,
  then  ${w_a = \min\{t-w_b-w_c, w_b+w_c\}}$ for every ${w \in \Out_a(t\calP_{G_{h,k}})}$
  and ${w_a = \max\{w_b-w_c, w_c-w_b\}}$ for every ${w \in \In_a(t\calP_{G_{h,k}})}$. \qed
\end{lem}

It is a consequence of Lemma~\ref{lem:sametypecosets} and Lemma~\ref{lem:reduction} that the value of
\[\volev{a}(t\calP_{G_{h,k}}) - \volod{a}(t\calP_{G_{h,k}})\] does not depend on the chosen leaf-edge $a$;
it depends only on the integers $t$, $h$ and $k$.
Thus, for conciseness, we may define \[\Delta_{h,k}(t) := \volev{a}(t\calP_{G_{h,k}}) - \volod{a}(t\calP_{G_{h,k}}),\] where
$a$ is any leaf-edge of $G_{h,k}$.  We now state and prove a lemma that 
supplies the inductive step used in the validation of Lemma~\ref{lem:rootedloopedtrees}.

\begin{lem}[the inductive step]\label{lem:inductivestep}
   Let $h$ and $k$ be integers such that  $h \geq 1$, $k \geq 0$, and $h+k \geq 3$.
   For every nonnegative even integer $t$,  we have that
   \[
     \Delta_{h,k}(t)
     \ = \ \left\{\begin{array}{ll}
                           \Delta_{1,1}(t) \times \Delta_{1,k-1}(t) & \mbox{if $h = 1$}, \\
                           \Delta_{2,0}(t) \times \Delta_{h-1,k}(t) & \mbox{if $h > 1$}. 
                                                    \end{array}\right.
   \]
\end{lem}

\begin{proof}
  Let $a$ be a leaf-edge of $G_{h,k}$.
  Let $q$ be the other end vertex of $a$.
  Because $h+k \geq 3$, there are two other distinct edges $b$ and $c$ of $G_{h,k}$ incident to $q$.
  We may assume that $a$ is a stem of $G_{h,k}$ and $b$ is either a leaf-edge of $G_{h,k}$ or is incident to a loop
  (Figure~\ref{fig:induction}(a)). 
  We shall conveniently delineate two edge disjoint subgraphs of $G_{h,k}$ having the vertex $q$ as root.
  The first  is the subgraph $B$ that has $q$ as its root and $b$ as its stem.
  The second is the subgraph $C$ having $q$ as its root and $c$ as a stem.
  If $b$ is a leaf-edge of $G_{h,k}$, then $B=G_{2,0}$ and $C=G_{h-1,k}$, otherwise $B=G_{1,1}$ and $C=G_{h,k-1}$ (Figure~\ref{fig:induction}(b)).
  If we define
  \[
  \Delta_B(t) := \volev{b}(t\calP_{B}) - \volod{b}(t\calP_{B}) \mbox{ \ and \ } 
  \Delta_C(t) := \volev{c}(t\calP_{C}) - \volod{c}(t\calP_{C}), 
  \]
  then the assertion of this lemma translates into the equality $\Delta_{h,k}(t) = \Delta_B(t) \times \Delta_C(t)$. 
\begin{figure}[hbt]
\begin{center}
  \scalebox{0.85}{\includegraphics{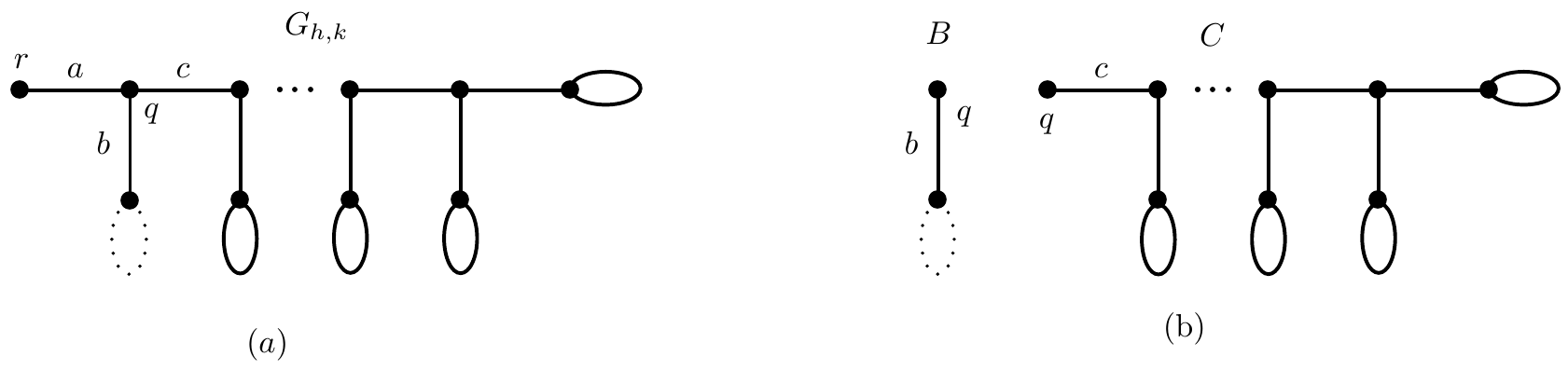}}
\end{center}  
  \caption{
  (a) $G_{h,k}$ with stem $a$, edges $b$ and $c$, and node $q$. The dotted loop may not exist.
  (b) The graph $B$ with stem $b$ and root $q$ and the graph $C$ with stem $c$ and root $q$.}\label{fig:induction}
\end{figure}

In order to prove the lemma it suffices to show that
\begin{align}
|\Out_a(t\calP_{G_{h,k}})| & = \volev{b}(t\calP_B) \times \volev{c}(t\calP_C)
                          + \volod{b}(t\calP_B) \times \volod{c}(t\calP_C), \mbox{ and} \label{eq:out}\\
|\In_a(t\calP_{G_{h,k}})| & = \volev{b}(t\calP_B) \times \volod{c}(t\calP_C) 
                    + \volod{b}(t\calP_B) \times \volev{c}(t\calP_C). \label{eq:in} 
\end{align}
Indeed, by setting $\calX = t\calP_{G_{h,k}}$ and $e = a$ in Lemma~\ref{lem:OI} of the
bijection by shifting, we attain as consequence that \newpage
\begin{align}
\Delta_{h,k}(t) & = |\Out_a(t\calP_{G_{h,k}})| - |\In_a(t\calP_{G_{h,k}})| \nonumber \\
               & = \volev{b}(t\calP_B) \times \volev{c}(t\calP_C)  + \volod{b}(t\calP_B) \times \volod{c}(t\calP_C) \nonumber \\
               & {\white = \ } - \volev{b}(t\calP_B) \times \volod{c}(t\calP_C)  - \volod{b}(t\calP_B) \times \volev{c}(t\calP_C) \nonumber \\
               & =  (\volev{b}(t\calP_B)- \volod{b}(t\calP_B)) \times (\volev{c}(t\calP_C)-\volod{c}(t\calP_C)) \nonumber \\ 
               & = \Delta_B(t) \times \Delta_C(t), \nonumber
\end{align}  
where the first equality is by Lemma~\ref{lem:OI},
the second equality is due to~\eqref{eq:out} and~\eqref{eq:in} and the last equality follows by definition.  

We start by showing Equation~\eqref{eq:out}.
For this we give maps between~$\Out_a(t\calP_{G_{h,k}})$ and the
collection of pairs~$(w^B,w^C)$ of integer points with~$w^B$ in~$t\calP_B$
and $w^C$ in $t\calP_C$ such that~$w^B_b + w^C_c$ is even,
indicating that $w^B_b$ and $w^C_c$ are both even or both odd.  
By Lemma~\ref{lemma:OeIe}, for every $w$ in $\Out_a(t\calP_{G_{h,k}})$ we have that
$w_a = \min\{t-w_b-w_c, w_b+w_c\}$,
implying that $w_b+w_c$ is even because~$t$ and $w_a$ are even.
Thus, the restriction of each $w$ in $\Out_a(t\calP_{G_{h,k}})$ to
$t\calP_B$ and to $t\calP_C$ produces a unique pair $(w^B,w^C)$ of integer
points such that $w^B_b + w^C_c$ is even.
Conversely, given such a pair $(w^B,w^C)$ of integers point with 
$w^B_b + w^C_c$ even,  we construct an integer point $w$ by setting
$w_e = w^B_e$ for $e \in E(B)$,
setting $w_e = w^C_e$ for~$e \in E(C)$,
and setting $w_a = \min\{t-w_b-w_c, w_b+w_c\}$.
Clearly~$w$ satisfies the system of inequalities that defines $\calP_{G_{h,k}}$
associated to each internal node $v$ in $G_{h,k}$ with $v \neq q$.
Let us now examine the inequalities associated to the vertex $q$.
Firstly, the value of $w_a$ was set so that $w_a + w_b + w_c \leq t$ and $w_a \leq w_b + w_c$.
Secondly, because ${w_c = w^C_c \geq 0}$ and ${w_b = w^B_b \geq 0}$,
if  $w_a = w_b+w_c$ then $w_b \leq w_a + w_c$ and $w_c \leq w_a + w_b$.
Finally, because $w_b = w^B_b \leq t/2$ and ${w_c = w^C_c \leq t/2}$,
if $w_a = t-w_b-w_c$, then $w_a + w_b =  t-w_b-w_c + w_b = t-w_c \geq w_c$ and
$w_a + w_c =  t-w_b-w_c + w_c = t-w_b \geq w_b$.
This concludes the proof of Equation~\eqref{eq:out}.

We verify the validity of Equation~\eqref{eq:in} in the same fashion.
We mount maps between~$\In_a(t\calP_{G_{h,k}})$ and the
collection of pairs~$(w^B,w^C)$ of integer points with~$w^B$ in~$t\calP_B$
and $w^C$ in $t\calP_C$ such that~$w^B_b + w^C_c$ is odd,
indicating that $w^B_b$ and $w^C_c$ have different parity.  
If $w$ is in $\In_a(t\calP_{G_{h,k}})$, then $w_a$ is odd and, by Lemma~\ref{lemma:OeIe}, 
${w_a = \max\{w_b-w_c, w_c-w_b\}}$, implying that $w_b+w_c$ is odd.
The restriction of $w$ to $t\calP_B$ and to $t\calP_C$
maps $w$ to a pair $(w^B,w^C)$ of integer points such that $w^B_b + w^C_c$ is odd.
This gives us one of the maps we want.
The other map is the extention of such a pair $(w^B,w^C)$ of integer points with
$w^B_b + w^C_c$ odd to the integer point $w$
such that $w_a = \max\{w^B_b-w^C_c, w^C_c-w^B_b\}$.
This gives us the converse map.
It is clear that $w$ satisfies the inequalities associated to each internal node $v$ of~$G_{h,k}$ with $v \neq q$.
By symmetry we may assume that $w_a = w_b - w_c = w^B_b - w^C_c$.
Because $0 \leq w^B_b \leq t/2$ and $0 \leq w^C_c \leq t/2$, then
$0 \leq w_a \leq t/2$,  $w_a \leq w_b + w_c$, $w_b = w_a + w_c$, $w_c = w_b - w_a \leq w_a + w_b$ and
$w_a + w_b + w_c = w^B_b - w^C_c + w^B_b + w^C_c = 2 w^B_b \leq t$. 
This concludes the proof of the lemma.
\end{proof}

We have now all the tools to prove Lemma~\ref{lem:rootedloopedtrees} of the evaluation
and close the proof of Theorem~\ref{thm:main} of the period of $L^{\calP}_G(t)$.

\newpage
\begin{proof}[Proof of Lemma~\ref{lem:rootedloopedtrees} (the evaluation)]
Using the abbreviation  $\Delta_{h,k}(t) = \volev{a}(t\calP_{G_{h,k}}) - \volod{a}(t\calP_{G_{h,k}})$,
we have to prove that
   \[
     \Delta_{h,k}(t)     \ = \ \left\{\begin{array}{ll}
                           (\frac{t}{2}+1)^{k} & \mbox{if $t = 0 \Mod 4$  or $h = 1$}, \\
                           0            & \mbox{if $t = 2 \Mod 4$}.
                                    \end{array}\right.
   \]
  The proof is by induction on $h+k$.
  The base case is $h+k = 2$ and we have two possibilities, either $(h,k) = (2,0)$ or $(h,k) = (1,1)$.
  Let us first consider the case  $(h,k) = (2,0)$.
  The system that defines $t\calP_{G_{2,0}}$ is simply 
  \begin{align}
  0 & \leq  w_a  \leq  t/2, \nonumber
  \end{align}
  where $a$ is the unique edge of $G_{2,0}$.
  The integer points in $t\calP_{G_{2,0}}$ are   $w_a = 0, 1, \ldots, t/2$.
  If~$t = 0 \Mod 4$ then $t/2$ is even and $\Delta_{2,0}(t) = 1$.
  If $t = 2 \Mod 4$ then $t/2$ is odd and~$\Delta_{2,0}(t) = 0$.
  Therefore,
   \begin{align}
     \label{eq:delta20}
     \Delta_{2,0}(t)   &  \ = \ \left\{\begin{array}{ll}
                           1 & \mbox{if $t = 0 \Mod 4$}, \\
                           0 & \mbox{if $t = 2 \Mod 4$}.
     \end{array}\right.
   \end{align}

  Now we consider the case $(h,k) = (1,1)$.
  The system that defines $t\calP_{G_{1,1}}$ is simply 
  \begin{eqnarray*}
    w_a & \leq & t - 2w_{\ell}\\
    w_a & \leq & 2w_{\ell} \\
    w_a & \geq & 0,
  \end{eqnarray*}
  where $a$ is the leaf-edge and $\ell$ is the loop of $G_{1,1}$.
  By Lemma~\ref{lemma:OeIe}, we have that $w_a = 0$ for every $w$ in $\In_a(t\calP_{G_{1,1}})$
  and that ${w_a = \min\{t-2w_{\ell}, 2w_{\ell}\}}$ for every $w$ in $\Out_a(t\calP_{G_{1,1}})$.
  Thus,  $\In_a(t\calP_{G_{1,1}}) = \emptyset$ and, for each~$w_{\ell}$ in $\{0,\ldots,t/2\}$, there is a unique point $w$ in $\Out_a(t\calP_{G_{1,1}})$.
  Therefore, by setting $\calX = t\calP_{G_{1,1}}$ and $e = a$ in Lemma~\ref{lem:OI} of the bijection by shifting,
  we attain as consequence that for every nonnegative even integer $t$
\begin{align}
\label{eq:delta11}
\Delta_{1,1}(t) & = |\Out_a(t\calP_{G_{h,k}})| - |\In_a(t\calP_{G_{h,k}})| = \frac{t}{2}+1. 
\end{align}

Now, we may assume that $h+k \geq 3$. If $h = 1$, then
for every nonnegative even integer~$t$
\[
\Delta_{h,k}(t)  = \Delta_{1,1}(t) \times \Delta_{1,k-1}(t) 
                = \big(\frac{t}{2}+1\big) \times \big(\frac{t}{2}+1\big)^{k-1} 
                = \big(\frac{t}{2}+1\big)^k, 
\]
where the first equality follows from Lemma~\ref{lem:inductivestep} and the second equality is due to \eqref{eq:delta11} and the induction hypothesis.

If $h > 1$, then for every nonnegative even integer~$t$
\[
\Delta_{h,k}(t)  = \Delta_{2,0}(t) \times \Delta_{h-1,k}(t)
\ = \ \left\{\begin{array}{ll}
                           1   \times \big(\frac{t}{2}+1\big)^k  = \big(\frac{t}{2}+1\big)^k & \mbox{if $t = 0 \Mod 4$}, \\
                           0   \times \Delta_{h-1,k}(t) = 0         & \mbox{if $t = 2 \Mod 4$},
       \end{array}\right.
\]
where the first equality follows from
Lemma~\ref{lem:inductivestep} and the second equality is due to~\eqref{eq:delta20}
and the induction hypothesis.
With this we finalize the proof of this lemma.
\end{proof}


\section{Further research directions}\label{sec:furtherconnexions}

In this section, we present various notable results on the polytopes of Liu and Osserman, obtained in the course of our investigations, that seem interesting on their own.
We also discuss a surprising relation between these polytopes and invariants used 
to distinguish 3-manifolds as well as an unexpected connection with arrangements of curves.

\subsection{Skeleton of the $\{1,3\}$-trees polytopes}

The next result was previously observed by Liu and Osserman~\cite[proof of Corollary~3.6]{LiuO2006}. 

\begin{lem}\label{lem:fulldim}
  If $T$ is a $\{1,3\}$-tree, then the polytope $\calP_T$ is full-dimensional. 
\end{lem}

\begin{proof}
  If $n$ is the number of degree~3 nodes in $T$, then we must prove that $\calP_T$ has dimension~${2n+1}$. 
  Let $E$ be the set of edges of $T$.  For an edge $e \in E$, let~$w^e = \frac13(\um_T - \um_e)$.  
  It is easy to check that $w^e \in \calP_T$ for every $e$.  Moreover, the set $\{w^e : e \in E\}$
  together with the origin forms a set of affinely independent vectors with $2n+2$ vectors. 
\end{proof}

Here is a consequence of Theorem~\ref{thm:thevertices}.

\begin{cor}\label{cor:BijectionEvenSetofLeaves}
  For every $\{1,3\}$-tree $T$, there is a bijection between vertices of $\calP_T$ and 
  subsets of the leaves of $T$ with an even number of leaves.  Thus, for every 
  $\{1,3\}$-tree~$T$ with~$m$ edges, $\calP_T$ has $2^{\frac{m+1}{2}}$ vertices.
\end{cor}

\begin{proof}
  Let $S$ denote an arbitrary subset of the leaves of $T$ such that $|S|$ is even. 
  As~$T$ is a tree, each such $S$ corresponds to exactly one collection of $|S|/2$ 
  disjoint leaf-paths in~$T$ whose ends are exactly the leaves in $S$.  The converse is also true: to each collection $H$ of disjoint leaf-paths in $T$, we can associate 
  the set of ends of the paths in $H$, and this set contains only leaves, and clearly an 
  even number of them. 

  Recall that a $\{1,3\}$-tree on $m$ edges has $\ell = \frac{m+3}{2}$ leaves.
  If $\ell$ is even, then the number of subsets $S$ of the leaves with $|S|$ even 
  is half of the total number of sets of leaves in~$T$, that is, half of $2^{\ell}$. 
  If $\ell$ is odd, then the number of subsets $S$ of the leaves with $|S|$ even 
  is also half of the total number of sets of leaves in~$T$. Indeed, it is the sum 
  of $\binom{\ell}{2i}$ for $i = 0,1,\ldots,(\ell{-}1)/2$, which is equal to the sum of 
  $\binom{\ell-1}{i}$ for $i=0,\ldots,\ell{-}1$, that is~$2^{\ell-1} = 2^{\frac{m+1}{2}}$. 
\end{proof}

Corollary~\ref{cor:BijectionEvenSetofLeaves} implies that there is a nontrivial bijection between vertices of the polytopes of different $\{1,3\}$-trees with the same number of edges. 

The symmetric difference between sets $A$ and $B$ is the set $A \bigtriangleup B = (A \cup B) \setminus (A \cap B)$.
For a vertex $w$ of $\calP_T$, let $H_w$ denote the collection of disjoint leaf-paths such that ${w = \frac12\um_{H_w}}$. 

\begin{thm}\label{thm:PTadj}
  Let $w$ and $w'$ be two distinct vertices of $\calP_T$.
  Then $w$ and $w'$ are adjacent in the 1-skeleton of $\calP_T$ 
  if and only if $H_w \bigtriangleup H_{w'}$ is a leaf-path.
\end{thm}

\begin{proof}
  Note that $H_w \bigtriangleup H_{w'}$ is a disjoint collection $\{P_1,\ldots,P_k\}$ of leaf-paths, with $k \geq 1$.  
  If $k = 1$, then let $L$ denote the set of edges incident to leaves of~$T$.  
  Let $L_w$ denote the edges of $H_w$ in $L$ and $L_{w'}$ denote the edges of $H_{w'}$ in~$L$.  
  Because $k = 1$, we have that 
  \begin{equation}\label{eq:dif2}
    |L_w \bigtriangleup L_{w'}| = 2.
  \end{equation}
  The hyperplane 
  $h(x) : \sum_{e \in L_w \cap L_{w'}} 2\,x_e + \sum_{e \in L \setminus (L_w \cup L_{w'})}  (1-2\,x_e) = |L|-2$ 
  is a supporting hyperplane of $\calP_T$, with $w$ and $w'$ being the only vertices of $\calP_T$ 
  in this hyperplane.  Indeed, for every $x \in \calP_T$, 
  \begin{eqnarray}
  \sum_{e \in L_w \cap L_{w'}} \!\!\! 2\,x_e \ + \sum_{e \in L \setminus (L_w \cup L_{w'})} \!\!\! (1-2\,x_e)
   & \leq & |L_w \cap L_{w'}| + |L \setminus (L_w \cup L_{w'})| \label{eq:hyper} \\ 
   & = & |L \setminus (L_w \bigtriangleup L_{w'})| = |L|-2, \nonumber
  \end{eqnarray}
  where the last equality is due to~\eqref{eq:dif2}. 
  Also, if $x$ is a vertex of $\calP_T$, then 
  inequality~\eqref{eq:hyper} is tight if and only if the set $L_x$ of leaves of $H_x$ is
  such that $L_w \cap L_{w'} \subseteq L_x \subseteq L_w \cup L_{w'}$.  
  There are exactly only two different such sets $L_x$ with $|L_x|$ even, namely, $L_w$ and $L_w'$. 
  Therefore, by Corollary~\ref{cor:BijectionEvenSetofLeaves}, inequality~\eqref{eq:hyper} is tight 
  only for vertices $w$ and $w'$, implying that $w$ and $w'$ are adjacent in the 1-skeleton of $\calP_t$.

  If $k > 1$, then we will show that the middle point $m$ of the segment $[w,w']$
  is a convex combination of other two vertices of $\calP_T$, and therefore $w$ and $w'$ are 
  not adjacent.
  Let~$u = \um_{H_w \bigtriangleup P_1}$ and $v = \um_{H_w \bigtriangleup (P_2 \cup \cdots \cup P_k)}$.  
  Since~$H_w \bigtriangleup P_1$ and $H_w \bigtriangleup (P_2 \cup \cdots \cup P_k)$ are disjoint collections of leaf-paths, 
  $u$ and $v$ are vertices by Theorem~\ref{thm:thevertices}.  Also, $u$ and $v$ are distinct from $w$ 
  and~$w'$, as~$k > 1$.  Now it is enough to note that $m = \frac{w+w'}2 = \frac{u+v}2$. 
\end{proof}

\begin{cor}
  Let $T$ be a $\{1,3\}$-tree with $\ell$ leaves.
  Then the degree of each vertex of the polytope $\calP_T$ is $\binom{\ell}{2}$. 
\end{cor}

\begin{proof}
  From Theorem~\ref{thm:PTadj}, every vertex of $\calP_T$ 
  has a neighbour for each leaf-path in $T$.
\end{proof}

\begin{question} 
  Is there a (combinatorial) characterization of the vertices and edges 
  of the 1-skeleton of $\calP_G$ for an arbitrary $\{1,3\}$-graph $G$?
\end{question}

\subsection{Symmetry of the $\{1,3\}$-trees polytopes}

We propose a family of involutive isometries of $\calP_T$ showing its high degree of symmetry.

\begin{thm}\label{thm:isometry}
  Let $T$ be a $\{1,3\}$-tree with $n$ degree 3 nodes and 
  let $H$ be a disjoint collection of leaf-paths in $T$.  
  We define the function $h_H : \IR^{2n+1} \longrightarrow \IR^{2n+1}$ by
   $$ h_H(w) =\left\{
   \begin{array}{ll}
   \frac12 - w_e & \text{if } e \in E(H),\\ 
         w_e & \text{otherwise}.\\
  \end{array}\right.$$   
  Then $h_H$ is an isometry of $\calP_T$ to itself. 
\end{thm}

\begin{proof} We notice that $h_H$ is an involution, and that $$h_H(w)=w \cdot B + \frac12\um_H$$
where $B$ is the $((2n+1)\times(2n+1))$-matrix where the entry $(e,e)_{1\le e\le 2n+1}$ equals~$1$ (resp.~$-1$) if $e\not\in E(H)$ (resp.\ $e\in E(H)$) and zero elsewhere. It can be checked that $det(B)=\pm 1$ and also that $B\cdot B^t = B^t\cdot B = I$ with $B^t$ the transpose of $B$ and $I$ the identity. Therefore, $B$ is a rotation matrix and thus an isometry.
The $\frac12\um_H$ translation does not affect the isometry, so $h_H$ is indeed an isometry.

Moreover, if $v$ is a vertex of $\calP_T$ then $h_H(v)$ is a vertex of $h_H(\calP_T) = \calP_T$. Indeed, if~$v$ is a vertex of $\calP_T$ then there is a collection of disjoint leaf-paths $H_v$ such that $v=\frac{1}{2}\um_{H_v}$. We notice that 
$$h_H(v) = \frac{1}{2}\um_{H \bigtriangleup H_v}$$
and, because $H \bigtriangleup H_v$ is also a collection of disjoints leaf-paths, $h_H(v)$ correspond to a vertex of $\calP_T$. Hence $h_H$ is an isometry of $\calP_T$ to itself.
\end{proof}

\begin{remark} 
Each function $h_H$ is an involution, and it can be therefore thought of as a particular {\em even permutation} on the set of vertices of $P_T$. 
\end{remark}

By Theorem~\ref{thm:thevertices}, we clearly have that $2\calP_T$ is a 0/1 polytope. In view of the above combinatorial properties, it might be reasonable to consider $2\calP_T$ as a good candidate to study different questions in connection with 0/1 polytopes.  For instance, a very basic (but difficult) problem is to count the minimal number of simplices needed to
triangulate the $d$-dimensional cube.  The following question is on the same spirit.

\begin{question} 
  Let $T$ be a $(1,3)$-tree. What is the smallest number of simplices needed to 
  triangulate $2\mathcal{P}_{T}$?
\end{question}

\subsection{Graphs with the same degree sequence}

Liu and Osserman~\cite[Remark~3.11]{LiuO2006} observed that, if $G$ and $H$ are two connected $\{1,3\}$-graphs on $n$ nodes and $m$ edges, then $G$ and $H$ have the same number of internally Eulerian subgraphs. In other words, for connected $\{1,3\}$-graphs, the number~$N_G$ depends only on the number of nodes and edges in~$G$.  Specifically, as we have mentioned just after Example~\ref{example}, if $k=m-n+1$ is the cyclomatic 
number of $G$ and $h$ is the number of leaves in $G$, then $N_G = 2^k$ if $h=0$ and $N_G = 2^{k+h-1}$ if $h>0$.

We observe that, if $G$ and $H$ are two connected graphs on $n$ nodes and the same degree sequence, then $G$ and $H$ have the same number of internally Eulerian subgraphs, that is, $N_G=N_H$.  Indeed, a previous result~\cite[Theorem~1]{FernandesPRAR2020} states that $G$ can be transformed into $H$ by a series of NNIs.  Specifically, an NNI move preserves the degree sequence of the graph.  The discussion that precedes Lemma~\ref{claim:NNI} establishes for connected $\{1,3\}$-graphs that an NNI move preserves (internally) Eulerian subgraphs; but NNI moves also preserve (internally) Eulerian subgraphs in connected graphs with the same degree sequence.  Thus the NNIs transforming $G$ into $H$ naturally induce a bijection between (internally) Eulerian subgraphs of~$G$ and~$H$.  From the proof of~\cite[Theorem~1]{FernandesPRAR2020}, we may strengthen Liu and Osserman's remark as follows.

\begin{remark}
  Let $G$ be a connected graph with $n$ nodes and $m$ edges. 
  Let $h$ be the number of leaves of $G$ and 
  $k=m-n+1$ be the cyclomatic number of $G$. 
  Then the number $N_G$ of internally Eulerian subgraphs of $G$ is $2^k$ 
  if~$h=0$ and $2^{k+h-1}$ if $h>0$.
\end{remark}

This leads also to a purely combinatorial proof (instead of a linear 
algebraic approach) that the number of Eulerian subgraphs of a connected 
graph with $n$ nodes and $m$ edges is $2^k$ where $k=m-n+1$ is the 
cyclomatic number of the graph~\cite[Theorem~1.9.6]{Diestel1997}. 

\subsection{Nonintersecting closed curves}

A \emph{3-regular hypergraph} is a pair $H = (V,E)$ where $V$ is the set of vertices of $H$ and each element of $E$ is a {\em hyperedge}, and consists of exactly three elements of $V$. 
Let $a$, $b$, and $c$ be the three vertices in a hyperedge $e \in E$. 
Let $w_a$, $w_b$, and $w_c$ be variables satisfying the following system of linear 
inequalities, which we refer to as~$S^{H}_t(e)$:
\begin{eqnarray}
  w_a & \leq & w_b + w_c \nonumber\\
  w_b & \leq & w_a + w_c \nonumber\\
  w_c & \leq & w_a + w_b \nonumber\\
  w_a + w_b + w_c & \leq & t\,. \label{eq:paritynoncross}
\end{eqnarray}

Let~$S^{H}_t$ be the union of all the linear systems $S^H_t(e)$ taken over all hyperedges~$e$
of~$H$ and let $\calP_H$ be the set consisting of all real solutions to this linear system 
when $t=1$. Because of the constraints~\eqref{eq:paritynoncross}, 
$\calP_H$ turns out to be a polytope. 

Given a cubic graph $G$, we can naturally associate a 3-regular hypergraph $H_G$ 
having as vertices the set of edges of $G$ and each hyperedge is given by the edges 
incident to a vertex of $G$. In this case, we have that $\calP_{H_G}=\calP_G$. 

For each hyperedge $e$ of a hypergraph $H$, let us consider an auxiliary variable~$z_e$ and 
substitute~\eqref{eq:paritynoncross} in each system~$S^H_t(e)$ by the~\emph{parity constraint}:
\begin{eqnarray*}
  w_a + w_b + w_c & = & 2\,z_e  \\
             z_e & \leq & t \, . \nonumber
\end{eqnarray*}

\newcommand{\PS}{\bar{S}}

Let $\PS^{H}_t$ be the union of all these modified linear systems, taken over all hyperedges~$e$ of~$H$. 
The polytope~$\calQ_H$ consists of all real solutions to this linear system when $t=1$. We notice that $\calQ_{H_G}=\calQ_{G}$ for every cubic graph~$G$.

Let $\calT$ be a triangulation of a 3-manifold. Let  $H_{\calT} = (V,E)$ be the hypergraph having as set $V$ of vertices the edges of $\calT$ and the set $E$ of hyperedges are the 3-sets corresponding to the triangles which are faces of the tetrahedra used in $\calT$. We notice that a vertex of~$H_{\calT}$ (that is, an edge of $\calT$) could belong to more than two hyperedges (that is, the corresponding edge is shared by two or more tetrahedra in $\calT$).  Maria and Spreer~\cite{MariaS2016} studied the notion of \emph{admissible colourings} of the edges of $\calT$ with $r-1$ colours, that correspond to integer solutions of the linear system~$\PS^{H_{\calT}}_t$ for $t=r-2$, and thus correspond to integer points in the dilated polytope~$t\calQ_{H_{\calT}}$.  They interpret each admissible colouring as a surface embedded in the triangulated 3-manifold and use this to derive better algorithms to compute Turaev-Viro invariants of degree~$4$ for the 3-manifold.
Note that Maria and Spreer~\cite[Section~2.3]{MariaS2016} defined a reduction of an admissible colouring that is an admissible colouring with two colours, that is, $r=3$, and they decomposed the invariants according to these reduced colourings.  This plays very much the same role as the cosets we used in our results. 

In the same spirit, we mimic the above construction of Maria and Spreer~\cite{MariaS2016} for a triangulation $T$ of the 2-sphere. 
We do so by taking $T$ as a graph embedded in the plane and by considering its dual graph $T^*$.  Note that $T^*$ is a planar cubic graph.  We shall see that the integer points in the dilated polytope $t\calQ_{T^*}$ have an intriguing geometric interpretation in terms of arrangements of pseudocircles.

A {\em pseudocircle} is a non-self-intersecting continuous closed curve in the plane.  
A \emph{$T$-arrangement of pseudocircles} is a (possibly empty) set of nonintersecting pseudocircles~$C$ on the plane 
such that \emph{(i)} $C$ intersects $T$ transversally in the interior of edges (not touching vertices) and \emph{(ii)} when $C$ enters into a facial triangle through an edge, it leaves the triangle through a different edge.
The \emph{order} of a $T$-arrangement is the maximum number of times 
a facial triangle of $T$ is traversed by pseudocircles in the arrangement. 


It turns out that each integer point in the rational polytope $t\calQ_{T^*}$ induces a $T$-arrangement of pseudocircles of order at most~$t$ and vice-versa.  Indeed, for each such integer point in $t\calQ_{T^*}$, 
we can construct systems of arcs in each facial triangle of~$T$ inducing such a $T$-arrangement of pseudocircles. To show this correspondence, we may proceed as follows. 

Consider a facial triangle of $T$ formed by edges $\{a,b,c\}$ 
and let $w'_a$, $w'_b$, and $w'_c$ be the values of $w_a$, $w_b$, and $w_c$ 
in the solution $\mathbf w$, respectively.
We may assume without loss of generality that $w'_a\ge w'_b \ge w'_c$.  
Draw $w'_a$ points along $a$, and similarly for $b$ and $c$.
Recall that $w'_a \leq w'_b+w'_c$.  Let $x_{ab}$ be the common vertex of $a$ and~$b$. 
If $w'_a = w'_b + w'_c$ then draw arcs joining the $w'_b$ points in $a$ closer 
to $x_{ab}$ to the points in~$b$, and draw arcs joining the remaining $w'_c$
points in $a$ to the points in $c$ (Figure~\ref{fig:configurations}(a)).
If~$w'_a < w'_b + w'_c$ then let $x_{bc}$ be the common vertex of $b$ and $c$. 
First draw arcs joining the $(w'_b+w'_c-w'_a)/2$ points in $b$ closer to $x_{bc}$ 
to the $(w'_b+w'_c-w'_a)/2$ points in $c$ closer to $x_{bc}$. 
Then draw arcs joining the remaining points in $b$ to the points in $a$ closer to $x_{ab}$,
and draw arcs joining the remaining points in $c$ to the points farther from $x_{ab}$ in $a$ (Figure~\ref{fig:configurations}(b)).
\begin{figure}[htb]
\begin{center}
\begin{tikzpicture}[LabelStyle/.style={inner sep=1pt},
  hl/.style={fill=green!30, circle, minimum size=15pt},nn/.style={}]

  \node[label={\small (a)}] at (-1,1) {};
  \node(v0) [black vertex] at (0,0) {};
  \node(v1) [black vertex] at (2,0) {};
  \node[label=above:{\small $x_{ab}$}](v2) [black vertex] at (1,1.5) {};
  \node[label={\small $c$}](c) at (1,-0.6) {};
  \node[label={\small $b$}](b) at (1.7,0.5) {};
  \node[label={\small $a$}](a) at (0.3,0.5) {};
  \Edge[color=black,lw=1.2pt](v0)(v1);
  \Edge[color=black,lw=1.2pt](v0)(v2);
  \Edge[color=black,lw=1.2pt](v1)(v2);

  \node[bb vertex](a5) at (0.310,0.45) {};
  \node[bb vertex](a4) at (0.405,0.6) {};
  \node[bb vertex](a3) at (0.5,0.75) {};
  \node[bb vertex](a1) at (0.595,0.9) {};
  \node[bb vertex](a2) at (0.690,1.05) {}; 

  \node[bb vertex](b2) at (1.432,0.87) {};
  \node[bb vertex](b1) at (1.57,0.637) {};

  \node[bb vertex](c1) at (0.77,0) {};
  \node[bb vertex](c2) at (1,0) {};
  \node[bb vertex](c3) at (1.23,0) {};

  \Edge[color=red,lw=1pt](a1)(b1);
  \Edge[color=red,lw=1pt](a2)(b2);
  \Edge[color=red,lw=1pt](a3)(c3);
  \Edge[color=red,lw=1pt](a4)(c2);
  \Edge[color=red,lw=1pt](a5)(c1);

  \node[label={\small (b)}] at (4,1) {};
  \node(v0) [black vertex] at (5,0) {};
  \node[label=right:{\small $x_{bc}$}](v1) [black vertex] at (7,0) {};
  \node[label=above:{\small $x_{ab}$}](v2) [black vertex] at (6,1.5) {};
  \node[label={\small $c$}](c) at (6,-0.6) {};
  \node[label={\small $b$}](b) at (6.7,0.5) {};
  \node[label={\small $a$}](a) at (5.3,0.5) {};
  \Edge[color=black,lw=1.2pt](v0)(v1);
  \Edge[color=black,lw=1.2pt](v0)(v2);
  \Edge[color=black,lw=1.2pt](v1)(v2);

  \node[bb vertex](a5) at (5.310,0.45) {}; 
  \node[bb vertex](a4) at (5.405,0.6) {};
  \node[bb vertex](a3) at (5.500,0.75) {};
  \node[bb vertex](a2) at (5.595,0.9) {};
  \node[bb vertex](a1) at (5.690,1.05) {}; 

  \node[bb vertex](b1) at (6.296,1.05) {};
  \node[bb vertex](b2) at (6.391,0.90) {};
  \node[bb vertex](b3) at (6.486,0.75) {};
  \node[bb vertex](b4) at (6.571,0.60) {};

  \node[bb vertex](c1) at (5.77,0) {};
  \node[bb vertex](c2) at (6,0) {};
  \node[bb vertex](c3) at (6.23,0) {};

  \Edge[color=red,lw=1pt](a1)(b1);
  \Edge[color=red,lw=1pt](a2)(b2);
  \Edge[color=red,lw=1pt](a3)(b3);
  \Edge[color=red,lw=1pt](a4)(c2);
  \Edge[color=red,lw=1pt](a5)(c1);
  \Edge[color=red,lw=1pt](b4)(c3);
\end{tikzpicture}
\end{center}
\caption{(a)~$w'_a=w'_b+w'_c$ and (b) $w'_a<w'_b+w'_c$.} 
\label{fig:configurations}
\end{figure}
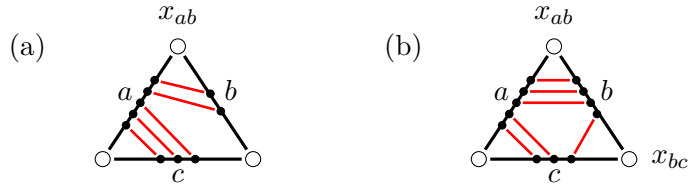

\begin{example}\label{example:curves} 
  Consider a triangulation of the 2-sphere whose corresponding graph is a~$K_4$.
  Note that the dual $K_4^*$ of a $K_4$ is isomorphic to $K_4$ (Figure~\ref{fig:pseudocircles}(a)). 
  It can be checked that
  $w_1=w_4=w_6=2, w_2=w_3=3, w_5=1, z_{e_1}=2, z_{e_2}=z_{e_3}=z_{e_4}=3$ and
  $w_1=w_2=w_3=w_4=w_5=w_6=2, z_{e_1}=z_{e_2}=z_{e_3}=z_{e_4}=3$ 
  are two integer points in $3\calQ_{K_4}$. 
  The corresponding induced nonintersecting pseudocircles are 
  illustrated in Figure~\ref{fig:pseudocircles}(b).

\begin{figure}[htb]
  \begin{minipage}[h]{0.5\textwidth}
    \scalebox{.5}{\includegraphics{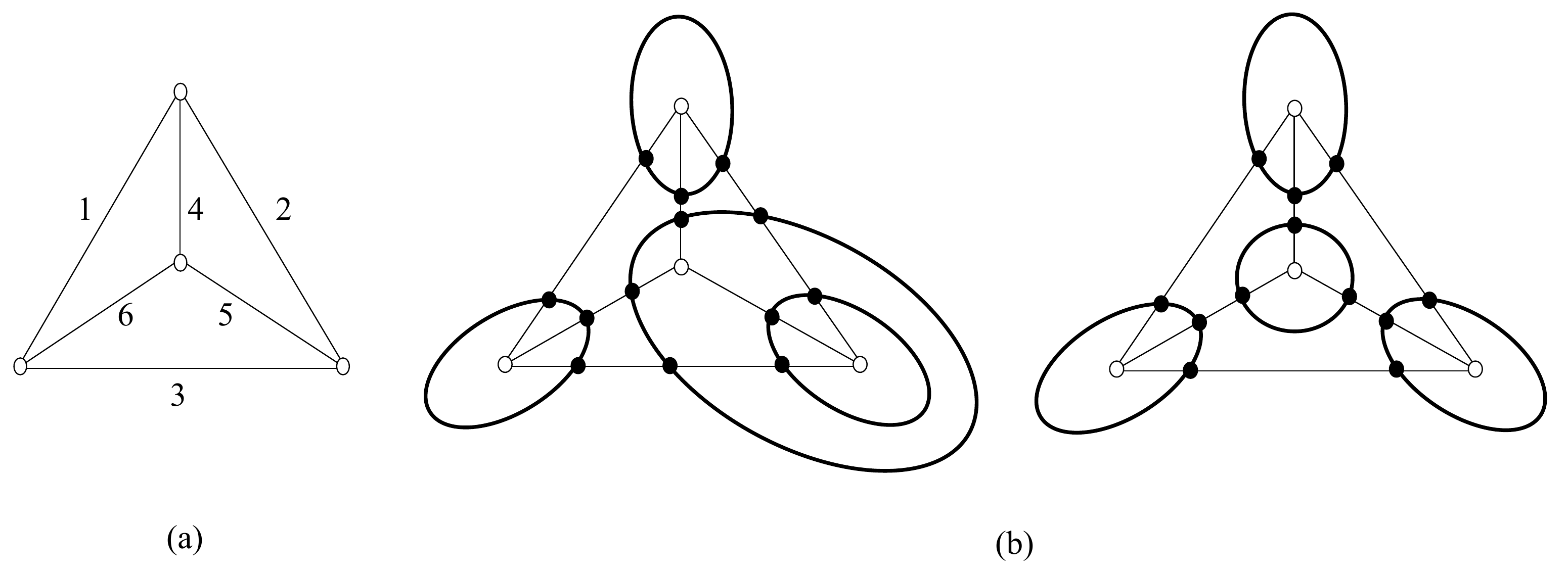}}
  \end{minipage}
  \caption{(a) $K_4$ with labeled edges, representing a triangulation of the plane; 
    (b) Two arrangements of curves of order 3.}
  \label{fig:pseudocircles}
\end{figure}
\end{example}

As a consequence of the above discussion we conclude that the number $L^{\calQ}_{T^*}(t)$ of integer points in $t\calQ_{T^*}$ is exactly the number of $T$-arrangements of pseudocircles of order at most $t$. 

\begin{lem}\label{lem:curves}
  For any triangulation $T$ of the 2-sphere, 
  the number of $T$-arrangements of pseudocircles of order at most~$t$ 
  is the number $L^{\calQ}_{T^*}(t)$ of integer points 
  in the rational polytope~$t\calQ_{T^*}$.
\end{lem}

We know by Ehrhart theory that the number of points in $t\calQ_{T^*}$ grows as a quasi-polynomial in~$t$, so that, by Lemma~\ref{lem:curves}, the number of $T$-arrangements of pseudocircles of order at most $t$ also grows as a quasi-polynomial in~$t$.  But we would like to know whether this quasi-polynomial collapses to a polynomial function of~$t$. 

\begin{question}
  Let $T$ be a triangulation of the plane and $t \geq 0$ be an integer. 
  Does the number of $T$-arrangements of pseudocircles of order $t$ grow polynomially in $t$?
\end{question}

A positive answer to the previous question would imply that $\calQ_{T^*}$ has period~1. 

\medskip 

A well-known problem in dimension 1 asks for the number of ways to construct an admissible set of $n$ parentheses for a word of length $2n$~\cite[Problem~6.19(b)]{Stanley1999}.  To put this problem into our context, this 1-dimensional counting problem is equivalent to the following question~\cite[Problem~6.19(o)]{Stanley1999}:

\begin{quote}
  \emph{What is the number of ways of connecting $2n$ points lying on a horizontal line by $n$ nonintersecting arcs, each arc connecting two of the points and lying above the points?} 
\end{quote}

\begin{figure}[htb]
  \centering
  \includegraphics[width=.95\linewidth]{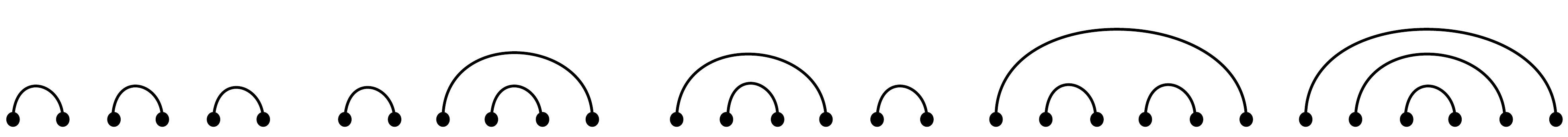}
  \caption{Five ways to connect 6 points by 3 nonintersecting arcs.}
  \label{fig:arcs}
\end{figure}

The answer to this problem is given by the Catalan numbers.  Note that each configuration in Figure~\ref{fig:arcs} naturally induces an arrangement of nonintersecting circles (by closing up arcs with their mirrors) which is closely connected to arrangements of pseudocircles.

\begin{question} 
Would the above information shed light on the understanding of $L_{\calQ_{T^*}}$?
\end{question}

\section*{Acknowledgements}

We would like to thank Arnaldo Mandel for helping us to achieve formula~\eqref{eq:recorrence}.

\bibliographystyle{plain}
\bibliography{CubicQGraphs}

\printglossaries
\end{document}